\documentclass[12pt,a4paper,reqno]{amsart}
\usepackage{amsmath}
\usepackage{amsfonts}
\usepackage{amssymb}
\usepackage{hyperref}
\usepackage{amsthm}

\numberwithin{equation}{section}

\theoremstyle{plain}

\newtheorem{theorem}{Theorem}[section]
\newtheorem{lemma}[theorem]{Lemma}

\newtheorem{cor}[theorem]{Corollary}
\newtheorem{conj}[theorem]{Conjecture}
\newtheorem{claim}[theorem]{Claim}
\newtheorem{prop}[theorem]{Proposition}
\theoremstyle{definition}
\newtheorem{example}[theorem]{Example}
\newtheorem*{observation}{Observation}

\newcommand{\eps}{\epsilon}

\newcommand{\al}{\alpha}
\newcommand{\be}{\beta}
\newcommand{\de}{\delta}
\newcommand{\el}{\ell}

\newcommand{\ga}{\gamma}
\newcommand{\la}{\lambda}

\newcommand{\Ga}{\Gamma}
\newcommand{\De}{\Delta}
\newcommand{\ka}{\kappa}

\newcommand{\bN}{\ensuremath{\mathbb{N}}}
\newcommand{\bZ}{\ensuremath{\mathbb{Z}}}
\newcommand{\bR}{\ensuremath{\mathbb{R}}}

\newcommand{\cP}{\ensuremath{\mathcal{P}}}
\newcommand{\cC}{\ensuremath{\mathcal{C}}}
\newcommand{\cI}{\ensuremath{\mathcal{I}}}
\newcommand{\cJ}{\ensuremath{\mathcal{J}}}
\newcommand{\cK}{\ensuremath{\mathcal{K}}}
\newcommand{\cT}{\ensuremath{\mathcal{T}}}
\newcommand{\cS}{\ensuremath{\mathcal{S}}}
\newcommand{\cF}{\ensuremath{\mathcal{F}}}
\newcommand{\SF}{\textnormal{SF}}
\newcommand{\dcup}{\mathbin{\mathaccent\cdot\cup}}

\newcommand{\card}[1]{\left| #1 \right|}
\newcommand{\floor}[1]{\left \lfloor #1 \right \rfloor}
\newcommand{\ceil}[1]{\left \lceil #1 \right \rceil}
\newcommand{\Mod}[1]{\ (\mathrm{mod}\ #1)}
\renewcommand{\t}[1]{\tilde{#1}}

\begin{document}
\title{On the structure of large sum-free sets of integers}
\author{Tuan Tran}

\address{Department of Mathematics, ETH, 8092 Zurich}
\email{manh.tran@math.ethz.ch}

\begin{abstract}
A set of integers is called sum-free if it contains no triple $(x,y,z)$ of not necessarily distinct elements with $x+y=z$. In this 
paper, we provide a structural characterisation of sum-free subsets of $\{1,2,\ldots,n\}$ of density at least $2/5-c$, where $c$ is an absolute positive constant.
As an application, we derive a stability version of Hu's Theorem [Proc.~Amer.~Math.~Soc. {\bf 80} (1980), 711--712] about the maximum size of a union of two sum-free sets in $\{1,2,\ldots,n\}$. We then use this result to show that the number of subsets of $\{1,2,\ldots,n\}$ which can be partitioned into two sum-free sets is $\Theta(2^{4n/5})$, confirming a conjecture of Hancock, Staden and Treglown [arXiv:1701.04754].
\end{abstract}

\maketitle

\section{Introduction}
A triple $(x,y,z)$ of not necessarily distinct integers is called a {\bf Schur triple} if $x+y=z$. Given a positive integer $r$, we say that a subset $A$ of $[n]:=\{1,2,\ldots,n\}$ is $r$-{\bf wise sum-free} if there exists an $r$-colouring of $A$ which contains no monochromatic Schur triples. When $r=1$, we simply call such sets {\bf sum-free}. 
Here we derive a structural theorem for large sum-free sets, and apply it to prove a sharp bound, up to a constant factor, on the number of $2$-wise sum-free subsets of $[n]$. 
In the following subsections we will review what is already known before presenting our results.

\subsection{Sum-free sets and their structure}
A natural extremal question, which was asked by Abbott and Wang \cite{Abbott-Wang77} in 1977, is how large an $r$-wise sum-free subset of $[n]$ can be. We denote the maximum by $\mu(n,r)$. It is not difficult to see that ${\mu(n,1)=n-\floor{n/2}}$, and this bound is attained by the set of odd numbers in $[n]$ and by the interval ${\{\floor{n/2}+1,\ldots,n\}}$. 
The following definition helps motivate the study of $\mu(n,r)$ for $r\ge 2$. Let $h(r)$ denote the largest positive integer $m$ for which there exists some way of partitioning $[m]$ into $r$ sets that are sum-free modulo $m+1$. For example, one has $h(2)=4$, $h(3)=13$ and $h(4)=44$ (see \cite[Problem I]{Abbott-Wang77}).
Abbot and Wang \cite{Abbott-Wang77} showed 
\[
\mu(n,r)\ge n-\floor{\frac{n}{h(r)+1}}\]
for every integer $r\ge 2$, and conjectured that the equality holds. In 1980, Hu \cite{Hu80} provided a short and elegant proof of this conjecture for $r=2$, that is, $\mu(n,2)=n-\floor{n/5}$.
To see why $\mu(n,2)\ge n-\floor{n/5}$, one can consider the set $\{a\in [n]:a\equiv 1, 4 \Mod{5}\} \cup \{b\in [n]:b\equiv 2, 3 \Mod{5}\}$. For $r\ge 3$, though there are several interesting general upper bounds for $\mu(n,r)$ (see \cite{Abbott-Wang77,HST17}), none of them matches the lower bound given by Abbot and Wang.

Given the extremal result, great efforts has been made to better understand the general structure of large sum-free sets in $[n]$. The earliest result in this direction was obtained by Freiman \cite{Freiman92} who showed that, loosely speaking, a sum-free set of density greater than $5/12$ is `interval like' or consists entirely of odd numbers.

\begin{theorem}[Freiman]\label{thm:Freiman-large-sum-free}
Every sum-free subset $A$ of
$[n]$ with $\card{A} \ge 5n/12+2$ satisfies one of the following conditions:
\begin{itemize}
\item[(i)] $A$ consists of odd numbers;
\item[(ii)] the minimum element of $A$ is at least $\card{A}$.
\end{itemize}	
\end{theorem} 
In an unpublished note, Deshouillers, Freiman and S\'os proved that the conclusion of Theorem \ref{thm:Freiman-large-sum-free} continues to hold when $5n/12+2$ is replaced by $2n/5+1$. The following examples show that the condition $\card{A}\ge 2n/5+1$ cannot be relaxed. Indeed, supposing that $n$ is divisible by $5$, we consider the sets $A_1=\{a\in [n]:a\equiv 1,4 \Mod{5}\}$, $A_2=\{a\in [n]:a\equiv 2,3 \Mod{5}\}$, and $A_3=\{n/5+1,\ldots, 2n/5\}\cup \{4n/5+1,\ldots,n\}$. We can see that each $A_i$ is a sum-free subset of $[n]$ of size $2n/5$, and that they are very far from satisfying property (i) or (ii) from Theorem \ref{thm:Freiman-large-sum-free}. 

A few years later, Deshouillers, Freiman, S\'os and Temkin \cite{DFST99} succeeded in slightly breaking the $2n/5$ barrier (see Theorem \ref{thm:DFST} below). Roughly speaking, they proved that the structure of a sum-free set in $[n]$ of size greater than $2n/5-O(1)$ is described by Theorem \ref{thm:Freiman-large-sum-free}, or close to one of the sets $A_i$ mentioned previously.\footnote{Their result provides no information about sum-free sets in $[n]$ of size less than $2n/5-\sqrt{n}$.}

\begin{theorem}[Deshouillers--Freiman--S\'os--Temkin]
\label{thm:DFST}
For every $x>0$, there exist numbers $n_0\in \bN$ and $K>0$ such that whenever $A$ is a sum-free set in $[n]$ satisfying $n \ge n_0$ and $|A| \ge 2n/5-x$, then $A$ has one of the following properties:
\begin{itemize}
		\item[(i)] all the elements of $A$ are odd;
		\item[(ii)] all the elements of $A$ are congruent to $1$ or $4$ modulo $5$;
		\item[(iii)] all the elements of $A$ are congruent to $2$ or $3$ modulo $5$;
		\item[(iv)] the minimum element of $A$ is greater than or equal to $\card{A}$;
		\item[(v)] $A$ is contained in $\left[\frac{n}{5}-K,\frac{2n}{5}+K\right] \cup \left[\frac{4n}{5}-K,n\right]$.
	\end{itemize} 
\end{theorem}

Besides being interesting in their own right, these results have found several applications (see \cite{Luczak95,Gr05,BLST15,BLST17}). 
We remark that very few structural results are known for large sum-free sets in finite abelian groups, cf. \cite{DT89,GrRu05,Lev05,DF06,Lev06,DL08,BPR16}.

\subsection{Counting sum-free sets}\label{sec:intro-counting}

Let $\SF_r(n)$ denote the collection of $r$-wise sum-free subsets of $[n]$. By considering all possible subsets of the set $\{\floor{n/2}+1,\ldots,n\}$, we see that $[n]$ contains at least $2^{n/2}$ sum-free sets. Cameron and Erd\H{o}s \cite{Cameron-Erdos} in 1990 conjectured that this trivial lower bound is within a constant factor of the truth, that is, $\card{\SF_1(n)}=O(2^{n/2})$. Their conjecture resisted various attempts at proof for over ten years \cite{Alon91,Calkin90, Freiman92}, until it was confirmed independently by Green \cite{Gr05} and Sapozhenko \cite{Sapozhenko03}. In fact, they proved that there are asymptotically $c(n)2^{n/2}$ such sets, where $c(n)$ takes two different constant values depending on the parity of $n$. Recently, a refinement of the Cameron--Erd\H{o}s conjecture was obtained by Alon, Balogh, Morris and Samotij \cite{ABMS14}, giving an upper bound on the number of sum-free sets in $[n]$ of size $s$, for all $s \in \{1,2,\ldots,\ceil{n/2}\}$.

For $r=2$, recall that the set $\{a\in [n]:a\equiv 1, 4 \Mod{5}\} \cup \{b\in [n]:b\equiv 2, 3 \Mod{5}\}$ is $2$-wise sum-free, and so are all of its subsets, giving $\card{\SF_2(n)}\ge 2^{4n/5}$. Inspired by \cite{Sapozhenko03,Gr05}, Hancock, Staden and Treglown \cite{HST17} considered this counting problem, among other things, and conjectured that this simple bound is in fact the correct estimate on $\card{\SF_2(n)}$. Thus they put forward the following conjecture.

\begin{conj}[Hancock--Staden--Treglown]\label{conj:Hancock-Staden-Treglown}
$\card{\SF_2(n)}=O(2^{4n/5})$.
\end{conj}

Note that Hancock et al. applied the {\bf container theorems} of  Balogh, Morris and Samotij \cite{BMS}, and Saxton and Thomason \cite{ST}, to establish $\card{\SF_2(n)}=2^{4n/5+o(n)}$. We recommend \cite{HST17,HT17} and the references therein for related results concerning $\mathcal{L}$-free subsets of $[n]$, where $\mathcal{L}$ is a homogeneous system of linear equations.

\subsection{Our results}
Here we go one step beyond Theorem \ref{thm:DFST}, and provide a structural characterisation of sum-free sets of size greater than $(2/5-c)n$, where $c$ is an absolute positive constant.
 
\begin{theorem}\label{thm:structure}
There exists an absolute positive constant $c$ so that the following holds for every $n\in \bN$ and every $\eta \in \bR$ with $2/n \le \eta \le c$. Let $A$ be a sum-free subset of $[n]$ with $\card{A} \ge (2/5-\eta)n$. Then one of the following alternatives occurs:
	\begin{itemize}
		\item[(i)] all the elements of $A$ are odd;
		\item[(ii)] all the elements of $A$ are congruent to $1$ or $4$ modulo $5$;
		\item[(iii)] all the elements of $A$ are congruent to $2$ or $3$ modulo $5$;
		\item[(iv)] the minimum element of $A$ is greater than or equal to $\card{A}$;
		\item[(v)] $A$ is contained in $\left[\left(\tfrac15-200\sqrt{\eta}\right)n,\left(\tfrac25+200\sqrt{\eta}\right)n\right] \cup \left[\left(\tfrac45-200\sqrt{\eta}\right)n,n\right]$.
	\end{itemize} 
\end{theorem} 

Note that there are sum-free subsets of $[n]$ of density $3/8$ structurally different from those appeared in the above theorem, such as $\{a \in [n]:a\equiv 3,4,5 \Mod{8}\}$ and $\{a \in [n]:a\equiv 4,5,6 \Mod{8}\}$. As an application of Theorem \ref{thm:structure}, we derive a stability version of Hu's result (Proposition \ref{prop:stability}), which may be of independent interest.

The proof of Theorem \ref{thm:structure} draws on a number of ideas from \cite{DFST99}. In particular, as in \cite{DFST99} we make use of an inverse theorem of Lev and Smeliansky \cite{Lev-Smeliansky95} for subsets of integers with small difference set. We also develop a number of new ideas in order to deal with the case when the smallest element of $A$ is sublinear in $n$, thereby making the argument substantially more involved. 

The second part of the paper deals with Conjecture \ref{conj:Hancock-Staden-Treglown}. We show 
\[
\card{\SF_2(n)}=O(2^{4n/5}),
\]
settling the conjecture in the affirmative. 
  
\begin{theorem}\label{thm:counting}
The number of $2$-wise sum-free subsets of $[n]$ is $O(2^{4n/5})$.
\end{theorem}

The proof technique is inspired by the methods of \cite{Cameron-Erdos,Gr05,ABMS14,ABMS,BLST17}. Among other tools we use a container lemma of Hancock et al. \cite{HST17}, an arithmetic removal lemma of Green \cite{Green05}, our stability version of Hu's theorem, and a recent bound on the number of sets of integers with small sumset due to Green and Morris \cite{GM16}.

\subsection{Organisation and notation} 
The rest is organised as follows. Section \ref{sec:structure} is devoted to the study of large sum-free subsets of $[n]$. In Section \ref{subsec:structure-main-lemmas} we provide the main lemmas and use them to obtain Theorem \ref{thm:structure}. We collect together some useful results in Section \ref{subsec:structure-inverse} and prove the main lemmas in Sections \ref{subsec:structure-large}, \ref{subsec:structure-middle} and \ref{subsec:structure-small}. Section \ref{sec:counting} deals with the enumerating problem.  In Section \ref{subsec:counting-overview}, we outline the proof of Theorem \ref{thm:counting}. We present the main tools in Section \ref{subsec:counting-lemmata} and prove Theorem \ref{thm:counting} in Section \ref{subsec:counting-containers}. We close, in Section \ref{sec:remarks}, with some remarks and open problems.

Given two non-empty sets $A, B \subset \bZ$, we define 
\[
A+B:=\{a+b:a\in A, b\in B\} \enskip \text{and} \enskip A-B:=\{a-b:a\in A,b\in B\}
\] 
to be their {\bf sumset} and {\bf difference set}, respectively.
For repeated addition we write $kA$ for the $k$-fold sumset $A+\ldots +A$, in contrast to $k\cdot A:=\{ka:a\in A\}$. For a finite set $A$ of integers, denote by $\min(A)$ and $\max(A)$ the minimum and maximum elements of $A$ respectively, and let $\el(A):=\max(A)-\min(A)+1$. Let $A_{+}$ stands for the set $\{a\in A: a>0\}$. The greatest common divisor of all the elements in $A-A$ will be denoted by $d(A)$.
We denote by $E$ the set of all even and by $O$ the set of all odd numbers in $[n]$; the value of $n$ will always be clear from the context. Denote 
\[
F_{1,4}=\{a\in [n]:a\equiv 1,4 \Mod{5}\} \enskip \text{and} \enskip F_{2,3}=\{a\in [n]:a\equiv 2,3 \Mod{5}\}.
\]
For real numbers $\al$ and $\be$, we employ the interval notation 
\[
[\al,\be]:=\{x \in \bZ:\al\le x\le \be\},
\]
and similarly for open intervals. Throughout the paper we omit floor and ceiling signs where the argument is unaffected.

\section{Large sum-free sets}\label{sec:structure}

\subsection{Main lemmas and a proof of Theorem \ref{thm:structure}}\label{subsec:structure-main-lemmas}
Here we state three main lemmas and explain how to obtain Theorem \ref{thm:structure} from them. In Lemma \ref{lem:large}, we deal with the sum-free sets $A$ for which the ratio $\frac{\min(A)}{\max(A)}$ is large. In Lemma \ref{lem:middle}, we deal with the case when the ratio $\frac{\min(A)}{\max(A)}$ is neither too large nor too small. Lemmas \ref{lem:large} and \ref{lem:middle} follow closely the approach from \cite{DFST99}, and only minor adaptations are needed in our setting. Finally in Lemma \ref{lem:small}, which is much more delicate, we study the case that the ratio $\frac{\min(A)}{\max(A)}$ is small. The methods used in \cite{DFST99} do not seem to adapt easily to this case, so we have been forced to devise our own arguments.

Our first main lemma, proven in Section \ref{subsec:structure-large}, says that if the ratio $\frac{\min(A)}{\max(A)}$ is large then $A$ satisfies condition $(v)$ from Theorem \ref{thm:structure}.

\begin{lemma}[Large range]\label{lem:large}
	Let $1/n \le \eta \le 1/160^2$, and let $A$ be a sum-free subset of $[n]$ such that $n \in A$, $d(A)=1$, $\card{A} \ge (2/5-\eta)n$, and 
	\[
	(1/5-\sqrt{\eta})n\le \min(A)<\card{A}.
	\]
	Then $A$ is contained in $\left[(\frac15-\sqrt{\eta})n,(\frac25+32\sqrt{\eta})n\right] \cup \left[(\frac45-31\sqrt{\eta})n,n\right]$.
\end{lemma}

Our second main lemma rules out the possibility that the ratio $\frac{\min(A)}{\max(A)}$ is neither too large nor too small. We provide the proof in Section \ref{subsec:structure-middle}.

\begin{lemma}[]\label{lem:middle}
	Let $1/n \le \eta \le 1/175^2$, and let $A$ be a sum-free subset of $[n]$ such that $n \in A$, $d(A)=1$, and 
	\[
	35\sqrt{\eta}n\le \min(A) \le (1/5-\sqrt{\eta})n.
	\]
	Then $\card{A} \le (2/5-2\eta)n$.
\end{lemma}

Our third and final main lemma, proven in Section \ref{subsec:structure-small}, states that if $\min(A)$ is small compared to $\max(A)$ then $A$ satisfies condition (i), (ii) or (iii) from Theorem \ref{thm:structure}.

\begin{lemma}[Small range]\label{lem:small}
	There exists an absolute positive constant $c$ such that the following holds for every $n\in\bN$ and every $\eta \in\bR$ with $1/n \le \eta\le c$. Let $A$ be a sum-free subset of $[n]$ satisfying $A\cap E \ne \emptyset$, $\card{A} \ge (2/5-\eta)n$, and
	\[
	\min(A) \le 35\sqrt{\eta}n.
	\]
	Then $A$ is contained in either $F_{1,4}$ or $F_{2,3}$.
\end{lemma}

With these lemmas in hand, we can prove Theorem \ref{thm:structure}.

\begin{proof}[Proof of Theorem \ref{thm:structure}]
Set $c=\min\left\{c_{\ref{lem:small}},\frac{1}{175^2}\right\}$, where $c_{\ref{lem:small}}$ is the absolute positive constant from Lemma \ref{lem:small}. Denote by $m$ and $N$ the minimum and maximum elements of $A$ respectively. We may assume without restriction of generality that $A\cap E\ne \emptyset$ and $\min(A)<\card{A}$, that is, $A$ does not satisfy properties (i) and (iv). In order to apply the main lemmas, we must show that $d(A)=1$ and $\eta\ge 1/N$. Suppose to the contrary that $d(A)>1$. Then there are two possibilities: either $d(A)=2$ or $d(A)\ge 3$. In the later case, we clearly have $\card{A} \le n/3+1$. In the former case, since $A\cap E \ne \emptyset$, $A$ consists of even numbers. In particular, the set $\{a/2:a \in A\}$ is a sum-free subset of $[n/2]$, and so $\card{A}\le n/4+1$. In either case, we always have $\card{A}\le n/3+1$, which contradicts the assumptions that $\card{A}\ge (2/5-\eta)n$ and $1/n\le \eta \le 1/175^2$. To verify the inequality $\eta \ge 1/N$, we note that $\card{A} \le N/2+1$ as $A$ is a sum-free subset of $[N]$. Since $\card{A} \ge (2/5-\eta)n$, this implies $N\ge n/2$ when $2/n \le \eta \le 1/175^2$, giving the required bound $\eta \ge 2/n \ge 1/N$. 

The proof now falls naturally into three cases:
\begin{equation*}
	(a) \hspace{0.2cm} (1/5-\sqrt{\eta})N \le m < \card{A} \quad 
	(b) \hspace{0.2cm} 35\sqrt{\eta}N \le m \le (1/5-\sqrt{\eta})N \quad (c) \hspace{0.2cm} m \le 35\sqrt{\eta}N.
\end{equation*}
We can easily rule out case $(b)$ using Lemma \ref{lem:middle}. If case $(c)$ occurs then Lemma \ref{lem:small} would imply that $A$ is a subset of either $F_{1,4}$ or $F_{2,3}$.
Finally we deal with case $(a)$. We may apply Lemma \ref{lem:large} to conclude  ${A\subseteq\left[(\frac15-\sqrt{\eta})N,(\frac25+32\sqrt{\eta})N\right]\cup \left[(\frac45-31\sqrt{\eta})N,N\right]}$. In particular, we have $\card{A} \le (2/5+65\sqrt{\eta})N$. This upper bound on $\card{A}$, in conjunction with the assumption that $\card{A} \ge (2/5-\eta)n$, shows $N \ge (1-163\sqrt{\eta})n$, which in turn implies $A\subseteq [(\frac15-200\sqrt{\eta})n,(\frac25+200\sqrt{\eta})n] \cup [(\frac45-200\sqrt{\eta})n,n]$ when $\eta \le 1/175^2$.
\end{proof}

\subsection{Inverse theorems}\label{subsec:structure-inverse}
Here we collect together a number of {\bf inverse theorems} that are essential for proving the main lemmas.

Sets with small sumset are a central object of interest in Arithmetic Combinatorics and have been extensively studied in recent years (see, for example, \cite{TaoVu06}). One of the main results in this area is Freiman's inverse theorem \cite{Freiman59} which states that if $A\subset\bZ$ and $\card{A+A} \le K\card{A}$ for some fixed $K$, then $A$ is a dense subset of a generalised arithmetic progression of bounded rank. In fact, the statement still holds in a slightly more general situation, when one considers $A+B$ instead of $A+A$. This was shown by Ruzsa \cite{Ruzsa94}. 
 
For relatively small $K$, one can obtain more precise information, which plays a crucial role in our study. It is not hard to see that for any finite and non-empty sets $A, B\subset \bZ$, one has 
\begin{equation}\label{Cauchy-Davenport}
	\card{A+B}\ge \card{A}+\card{B}-1,
\end{equation}
with equality if and only if $A$ and $B$ are arithmetic progressions with the same step. There has been much  on generalising this result. For instance, Lev and Smeliansky \cite{Lev-Smeliansky95} proved the following theorem.

\begin{lemma}[Lev--Smeliansky]\label{thm:Lev-Smeliansky}
	Let $A$ and $B$ be two finite sets of integers such that $\card{A+B}\le \card{A}+\card{B}+\min(\card{A},\card{B})-4$. Then $A$ is contained in an arithmetic progression of length $\card{A+B}-\card{B}+1$ and $B$ is contained in an arithmetic progression of length $\card{A+B}-\card{A}+1$, where both progressions have the same step.
\end{lemma}

The special case of the above result for $A=B$ is the famous Freiman's $3k-4$ theorem \cite{Freiman59}. For our investigation we shall, however, need a ``difference version'' of this theorem, which follows readily from Lemma \ref{thm:Lev-Smeliansky}.

\begin{lemma}\label{lem:Lev-Smeliansky}
	Let $A$ be a finite set in $\bZ$ such that $d(A)=1$. Then
	\[
	\card{(A-A)_{+}} \ge \min\left\{\tfrac12(\card{A}+\el(A)-2), \tfrac32\card{A}-2\right\}.
	\]   
\end{lemma}
\begin{proof}
Suppose for a contradiction that $\card{(A-A)_{+}}<\min\left\{\tfrac12(\card{A}+\el(A)-2), \tfrac32\card{A}-2\right\}$. As $\card{(A-A)_{+}} \ge \tfrac12\card{A-A}-\tfrac12$, it follows that $\card{A-A} \le \min\{\card{A}+\el(A)-2,3|A|-4\}$. By Lemma \ref{thm:Lev-Smeliansky}, we learn that $A$ is contained in an arithmetic progression of length $\card{A-A}-\card{A}+1 \le \el(A)-1$. This implies $d(A)>1$, a contradiction. 
\end{proof}

To our knowledge, the only extension of the $3k-4$ Theorem that applies to any set $A\subset \bZ$ with $\card{A+A}=3\card{A}+o(\card{A})$ was accomplished by Jin \cite{Jin07}. His proof is a tour de force of non-standard analysis.

\begin{lemma}[Jin]\label{thm:Jin}
	There exist an absolute positive constant $c$ and a natural number $K$ such that for every finite set $A$ of integers with $\card{A}>K$ and $\card{A+A}=3\card{A}-3+r$ for some integer $r$ with $0\le r\le c\card{A}$, $A$ satisfies at least one of the following properties:
	\begin{itemize}
		\item[(i)] $A$ is a subset of an arithmetic progression of length $2\card{A}-1+2r$;
		\item[(ii)] $A \subseteq P_1\cup P_2$ for some arithmetic progressions $P_1,P_2$ with common step and $\card{P_1}+\card{P_2} \le \card{A}+r$.   
	\end{itemize}
\end{lemma}

\subsection{Large range}\label{subsec:structure-large}
Here we give the proof of Lemma \ref{lem:large}. We shall need a simple but crucial observation from \cite[Proposition 2.1]{DFST99}. Its proof can be found in the appendix.

\begin{lemma}\label{lem:long-interval}
	Let $A$ be a sum-free set of positive integers and let $m$ be an arbitrary element of $A$. Then $A$ satisfies the following conditions:
	\begin{itemize}
		\item[(i)] $\card{A\cap \left([u,v]\cup [u+m,v+m]\right)} \le v-u+1$ for all $u, v \in \bN$ with $u\le v$;
		\item[(ii)] $\card{A\cap [u,u+2m-1]} \le m$ for every  $u \in \bN$;
		\item[(iii)] $\card{A \cap [u,v]} \le \tfrac12(v-u+m+1)$ for all $u,v \in \bN$ with $u\le v$.
	\end{itemize} 
\end{lemma}

We emphasise that in the first condition, the two intervals $[u,v]$ and $[u+m,v+m]$ are not necessarily disjoint.

We are now in position to prove Lemma \ref{lem:large}.

\begin{proof}[Proof of Lemma \ref{lem:large}]
Throughout the proof let $m$ denote the minimum element of $A$. In the first step, we show that $m$ is not much larger than $n/5$. 

\begin{claim}\label{claim:large}
$m \le (1/5+15\eta)n$.
\end{claim}

\begin{proof}
Suppose to the contrary that $m>(1/5+15\eta)n$. As $m\in A$, we may apply Lemma \ref{lem:long-interval} (ii) to $u=n-2m+1$ and obtain 
\begin{equation}\label{eq:large-M-2m-M}
\card{A\cap (n-2m,n]} \le m.
\end{equation}
Since $\card{A}>m$ by the assumption, this gives $m=\min(A) \le n-2m$, and so $m\le n/3$. One thus has 
\begin{equation}\label{eq:large-assumption}
(1/5+15\eta)n\le m \le n/3.  
\end{equation}
It follows from \eqref{eq:large-assumption} that $[m,n]$ is covered by the intervals $\left[m,\frac12(n-m)\right]$, $\left(\frac12 (n-m),\frac12 n\right]$, $(n-2m,n]$ and $ \left[m+m,m+\frac12(n-m)\right]$; so also is $A$.\footnote{One may verify this claim for $n/5 \le m \le n/4$, and for $n/4\le m \le n/3$ separately.} For the remainder of the proof we shall use this information to bound $\card{A}$.

Applying Lemma \ref{lem:long-interval} (i) with $u=m$ and $v=\tfrac12(n-m)$ we find 
\begin{equation}\label{eq:large-two-intervals}
\card{A\cap \left(\left[m,\tfrac12(n-m)\right] \cup \left[m+m,m+\tfrac12(n-m)\right]\right)} \le n/2-3m/2+1.
\end{equation}

We next bound $\card{A\cap\left(\frac12 (n-m),\frac12 n\right]}$. For abbreviation, let $B=A\cap \left(\frac12 (n-m),\frac12 n\right]$. Using \eqref{Cauchy-Davenport} gives
\begin{equation*}
\card{B} \le \card{2B}/2+1/2.
\end{equation*}

To estimate $\card{2B}$, we first observe that $2B \subseteq [n-m+1,n]$ as $B\subseteq \left(\frac12 (n-m),\frac12 n\right]$, and $(A-A)_{+} \subseteq [n-m]$ since $A\subseteq [m,n]$. Moreover since $A$ is sum-free and $B\subseteq A$, we must have $A\cap 2B=\emptyset$ and $A\cap (A-A)_{+}=\emptyset$. Hence $2B, (A-A)_{+}$ and $A$ are disjoint subsets of $[n]$, resulting in  
\begin{equation*}
\card{2B} \le n-\card{A}-\card{(A-A)_{+}}.
\end{equation*}

Note that $d(A)=1$ by the assumption, and $n-m \le 2\card{A}-3$ by \eqref{eq:large-assumption} and the assumption that $\card{A} \ge (2/5-\eta)n$ and $\eta \ge 1/n$. Lemma \ref{lem:Lev-Smeliansky} then implies
\begin{equation*}
\card{A}+\card{(A-A)_{+}} \ge \min\left\{\tfrac32\card{A}+\tfrac12(n-m-1),\tfrac52\card{A}-2\right\}=\tfrac32\card{A}+\tfrac12(n-m-1).
\end{equation*}

Assembling all the information, we get
\begin{equation}\label{eq:large-nm2n2}
\card{A\cap \left(\tfrac12 (n-m),\tfrac12 n\right]} \le n/4 + m/4 - 3\card{A}/4+3/4.
\end{equation}

Recalling that $A$ is covered by the intervals $\left[m,\frac12(n-m)\right]$, $\left(\frac12 (n-m),\frac12 n\right]$, $(n-2m,n]$ and $\left[m+m,m+\frac12(n-m)\right]$, and using estimates \eqref{eq:large-M-2m-M}, \eqref{eq:large-two-intervals} and \eqref{eq:large-nm2n2}, we deduce that $\card{A}\le 3n/4-m/4-3\card{A}/4+7/4$. Since $\card{A}\ge (2/5-\eta)n$ and $\eta \ge 1/n$ by the assumption, this leads to $m\le 3n-7\card{A}+7 \le (1/5+14\eta)n$, which contradicts our hypothesis that $m\ge (1/5+15\eta)n$.
\end{proof}

In the second step, we establish an approximate version of the lemma.

\begin{claim}\label{claim:large-approximate}
All integers in $\left[(\frac15+\sqrt{\eta})n,(\frac25-\sqrt{\eta})n\right]\cup \left[(\frac45+\sqrt{\eta})n,n\right]$, with at most $14\sqrt{\eta}n$ exceptions, belong to $A$.
\end{claim}

Before proving Claim \ref{claim:large-approximate}, we shall use it to finish the proof of Lemma \ref{lem:large}. Suppose to the contrary that $A \nsubseteq \left[(\frac15-\sqrt{\eta})n,(\frac25+32\sqrt{\eta})n\right] \cup \left[(\frac45-31\sqrt{\eta})n,n\right]$. Then there exists
$a \in A \cap \left[(\frac25+32\sqrt{\eta})n,(\frac45-31\sqrt{\eta})n\right]$
since $\min(A) \ge (1/5-\sqrt{\eta})n$ by the assumption. From this we get 
\[
a+(1/5+\sqrt{\eta})n\le (1-30\sqrt{\eta})n, \ \text{and} \ a+(2/5-\sqrt{\eta})n \ge (4/5+31\sqrt{\eta})n,
\] 
showing that the intervals $a+\left[(\frac15+\sqrt{\eta})n, (\frac25-\sqrt{\eta})n\right]$ and $\left[(\frac45+\sqrt{\eta})n,n\right]$ have at least $\min\{29\sqrt{\eta}n,(1/5-3\sqrt{\eta})n\}=29\sqrt{\eta}n$ elements in common. Thus, using pigeonhole principle and Claim \ref{claim:large-approximate}, we find $a+b=c$ for some $b,c \in A$, which contradicts the assumption that $A$ is sum-free.
\end{proof}

Finally we give a proof of Claim \ref{claim:large-approximate} using Claim \ref{claim:large}, Lemmas \ref{lem:Lev-Smeliansky} and \ref{lem:long-interval}.

\begin{proof}[Proof of Claim \ref{claim:large-approximate}]
As $(1/5-\sqrt{\eta})n \le m \le (1/5+15\eta)n$ and $\eta \le 1/160^2$ by Claim \ref{claim:large} and the assumption, we have the following chain of inequalities:
\begin{equation}\label{large:approximate-chain}
m \le \tfrac12 (n-m) \le \tfrac12 n \le n-2m \le n-m \le n.
\end{equation}
We shall use \eqref{large:approximate-chain} to prove the claim which, roughly speaking, states that 
\[
A \approx \left[m,\tfrac12(n-m)\right]\cup (n-m,n].
\]
(Note that $m\approx \tfrac15 n$, $\tfrac12 (n-m) \approx \tfrac25 n$, $n-2m \approx \tfrac35 n$ and $n-m\approx \tfrac45 n$.)
	
Since $d(A)=1$, it follows from Lemma \ref{lem:Lev-Smeliansky} that 
\[
\card{A}+\card{(A-A)_{+}} \ge \min\left\{\tfrac32\card{A}+\tfrac12(n-m-1),\tfrac52\card{A}-2\right\} \ge (1-10\eta)n,
\]
where the last inequality holds since $m \le (1/5+15\eta)n$ by Claim \ref{claim:large}, and $\card{A}\ge (2/5-\eta)n$ and $\eta \ge 1/n$ by the assumption. Moreover, as $A$ is a sum-free subset of $[m,n]$, $A\cap [n-m]$ and $(A-A)_{+}$ are disjoint subsets of $[n-m]$. Hence
\begin{align}\label{eq:large-approximate-nmn}
\card{A\cap (n-m,n]}& \ge \card{A}+\card{(A-A)_{+}}-\card{[n-m]} \ge m-10\eta n.
\end{align}

Since $A$ is sum-free, $(2\cdot A) \cap (n-m,n]$ and $A\cap (n-m,n]$ are disjoint, which gives
\begin{equation*}
\card{A\cap \left(\tfrac12 (n-m),\tfrac12 n \right]}= \card{(2\cdot A) \cap (n-m,n]}\le \card{(n-m,n]\setminus A} \overset{\eqref{eq:large-approximate-nmn}}{\le} 10\eta n.
\end{equation*}
We know from \eqref{large:approximate-chain} that $\card{A\cap\left([m,\tfrac12(n-m)] \cup (\tfrac12 n,n-2m]\right)}$ is at least 
	\begin{equation}\label{eq:large-approximate-mnm2}
	\card{A}-\card{A\cap \left(\tfrac12 (n-m),\tfrac12 n \right]}-\card{A\cap (n-2m,n]}\ge (2/5-11\eta)n-m
	\end{equation}
as $\card{A}\ge (2/5-\eta)n$ by the assumption, $\card{A\cap \left(\tfrac12 (n-m),\tfrac12 n \right]} \le 10\eta n$ by the previous estimate, and $\card{A\cap (n-2m,n]} \le m$ by Lemma \ref{lem:long-interval} (ii).
	
	We next apply Lemma \ref{lem:long-interval} (i) with $u=\tfrac12 n-m$ and $v=n-3m$ to obtain 
\begin{equation}\label{eq:large-approximate-nmn3m}
\card{A\cap \left((\tfrac12 n-m,n-3m]\cup (\tfrac12 n, n-2m]\right)} \le n/2-2m+1.
\end{equation} 
Using \eqref{large:approximate-chain} once again, we may bound $\card{A\cap [m,\tfrac12 n-m]}$ from below by
	\begin{gather*}
\card{A\cap\left([m,\tfrac12(n-m)] \cup (\tfrac12 n,n-2m]\right)}-\card{A\cap \left((\tfrac12 n-m,n-3m]\cup (\tfrac12 n, n-2m]\right)}\\
(\text{by \eqref{eq:large-approximate-mnm2} and \eqref{eq:large-approximate-nmn3m}}) \quad \le -(1/10+11\eta)n+m-1.
	\end{gather*}
This implies $\card{2A\cap[2m,n-2m]} \ge -(1/5+22\eta)n+2m-3$, due to \eqref{Cauchy-Davenport}. Moreover, since $m \ge (1/5-\sqrt{\eta})n$ and $\eta \le 1/160^2$ by the assumption, one has $(\tfrac12 n,n-2m] \subseteq [2m,n-2m]$. We thus get
$\card{(\tfrac12 n,n-2m]\setminus 2A} \le \card{[2m,n-2m]\setminus 2A}\le (6/5+22\eta)n-6m+4$. From this and the assumption that $2A\cap A=\emptyset$, we obtain 
\begin{equation}\label{eq:large-approximate-n2n2m}
\card{A\cap (\tfrac12 n,n-2m]} \le \card{(\tfrac12 n,n-2m]\setminus 2A}\le (6/5+22\eta)n-6m+4
\end{equation}

Clearly we can bound $\card{\left[m,\tfrac12(n-m)\right]\setminus A}$ from above by  
	\begin{align}\label{eq:large-approximate-m2nm} \notag
	 &\card{\left[m,\tfrac12(n-m)\right]}+\card{A\cap (\tfrac12 n,n-2m]}-\card{A\cap \left([m,\tfrac12(n-m)] \cup (\tfrac12 n,n-2m]\right)} \\
	& \quad \quad (\text{by \eqref{eq:large-approximate-n2n2m} and \eqref{eq:large-approximate-mnm2}}) \quad \le (13/10+33\eta)n-13m/2+5 \le 13 \sqrt{\eta}n
	\end{align}
	assuming $m\ge (1/5-\sqrt{\eta})n$ and $1/n \le \eta \le 1/160^2$.
	
	From \eqref{eq:large-approximate-m2nm} and \eqref{eq:large-approximate-nmn} we see that all elements of $\left[m,\tfrac12(n-m)\right]\cup (n-m,n]$ belong to $A$, with $(13\sqrt{\eta}+10\eta)n \le 14\sqrt{\eta}n$ exceptions. As $\left[m,\tfrac12 (n-m)\right] \cup (n-m,n]$ contains $\left[(\tfrac15+\sqrt{\eta})n,(\tfrac25-\sqrt{\eta})n\right]\cup \left[(\tfrac45+\sqrt{\eta})n,n\right]$ when $(1/5-\sqrt{\eta})n \le m \le (1/5+15\eta)n$ and $\eta \le 1/160^2$, the claim follows.
\end{proof}

\subsection{Middle range}\label{subsec:structure-middle}
Our goal is to prove Lemma \ref{lem:middle}. For this purpose, we shall require the following variant of a fairly simple result due to Deshouillers et al. \cite[Lemma 2.3]{DFST99}. We provide the proof in the appendix for completeness of exposition.

\begin{lemma}\label{lem:summation}
	Let $k\in \bN$ and $\eps \ge 0$, and let $A \subseteq [0,k-1]$ and 
	\[
	B=\{b_1<\ldots<b_{\el}\}
	\]
	be two sets of integers such that $\card{A} \ge (1-\eps)k$ and $b_{i+1}-b_i \le k$ for each $i \in [\el-1]$. Then
	\[
	\card{A+B} \ge (1-4\eps)(k+\el(B)).
	\]
\end{lemma}

We are now able to prove Lemma \ref{lem:middle}.

\begin{proof}[Proof of Lemma \ref{lem:middle}]
Throughout the proof let $m$ denote the minimum element of $A$.
Suppose to the contrary that $\card{A} \ge (2/5-2\eta)n$. 
Since $A$ is a sum-free subset of $[n]$, we thus have
\begin{equation}\label{eq:middle-aa-upper-bound}
\card{(A-A)_{+}} \le n-\card{A} \le \left(3/5+2\eta\right)n.
\end{equation}
To get a contradiction we seek to show $\card{(A-A)_{+}}\ge (3/5+3\eta)n$. The following claim serves as an intermediate step.

\begin{claim}\label{claim:middle} Let $\eps=(5\eta n+2)/m$, then we have
\begin{itemize}
	\item[(i)] $\card{A \cap [n-m+1,n]} \ge (1-\eps)m$,
	\item[(ii)] $\card{(A-A)_{+}\cap [m-1]} \ge (1-\eps)m-1$.
\end{itemize}	
\end{claim}
\begin{proof}
(i) As $d(A)=1$ and $\el(A)=n-m+1$, Lemma \ref{lem:Lev-Smeliansky} gives
\[
\card{A}+\card{(A-A)_{+}} \ge \min\left
\{\tfrac32\card{A}+\tfrac12(n-m-1),\tfrac52\card{A}-2\right\}\ge (1-5\eta)n-2
\]
for $m \le n/5$ and $\card{A} \ge (2/5-2\eta)n$. 
Moreover $A\cap [n-m]$ and $(A-A)_{+}$ are disjoint subsets of $[n-m]$ since $A$ is sum-free set in $[m,n]$. Therefore, we have 
\[
\card{A \cap [n-m+1,n]} \ge \card{A}+\card{(A-A)_{+}}-\card{[n-m]} \ge m-(5\eta n+2)=(1-\eps)m.
\]

\noindent (ii) It follows from (i) that
\[
\card{(A-A)_{+}\cap [m-1]} \ge \card{A\cap [n-m+1,n]}-1\ge (1-\eps)m-1. \qedhere
\]
\end{proof}

In the final step, we shall bound $\card{(A-A)_{+}}$ from below.

\begin{claim}\label{claim:middle-aa} 
$\card{(A-A)_{+}} \ge (3/5+3\eta)n$.	
\end{claim}
\begin{proof}
Let $\{a_1<a_2<\ldots<a_k\}$ be the set consisting of all elements $a \in A\cap [n-2m]$ such that $A \cap [a-m+1,a-1]=\emptyset$. Denote $a_{k+1}=n-2m+1$, and   $A_{i}=A\cap [a_{i},a_{i+1})$ for $i \in [k]$.
It is not difficult to see that the following holds:\\
$(\ast)$ For each $i\in [k]$, the gap between any two consecutive elements of $A_i$ is less than $m$.

Let $D=A\cap[n-m+1,n]$, and set $\eps=(5\eta n+2)/m$. From Claim \ref{claim:middle} (i) we have $\card{D} \ge (1-\eps)m$. Moreover, property $(\ast)$ implies that we may apply Lemma \ref{lem:summation} to $A=D-(n-m+1)$ and $B=-A_i$, obtaining
\begin{equation}\label{eq:middle-d-a}
\card{D-A_i} \ge (1-4\eps)(m+\el(A_i)) \quad \text{for every $ i\in [k]$}.
\end{equation}
Moreover, using parts (ii) and (iii) of Lemma \ref{lem:long-interval} yields 
\begin{equation}\label{eq:middle-a-upper-bound}
\card{A} =\card{A\cap [n-2m+1,n]}+\sum\card{A_i}\le m+\tfrac12 \sum(m+\el(A_i)).
\end{equation}
Furthermore, we can infer from property $(\ast)$ that $(A-A)_{+}\cap [m-1], D-A_1, \ldots, D-A_k$ are disjoint subsets of $(A-A)_{+}$. So
\begin{align}\label{middle-aa}
\card{(A-A)_{+}} &\ge \card{(A-A)_{+}\cap [m-1]}+\sum\card{D-A_i} \notag\\
(\text{by Claim \ref{claim:middle} (ii), \eqref{eq:middle-d-a}}) \quad &\ge (1-\eps)m-1+(1-4\eps)\sum (m+\el(A_i)) \notag\\ 
(\text{by \eqref{eq:middle-a-upper-bound}}) \quad & \ge (1-\eps)m-1+(1-4\eps)(2\card{A}-2m) \notag\\
\quad & = (-1+7\eps)m+(2-8\eps)\card{A}-1.
\end{align}
Observe that $\eps=(5\eta n+2)/m\le \min\{\tfrac15 \sqrt{\eta},\frac16\}$ since $35\sqrt{\eta}n\le m$ and $1/n\le \eta$ by the assumption. Combining this with the assumption that $m\le \left(1/5-\sqrt{\eta}\right)n$ and $\card{A} \ge \left(2/5-2\eta\right)n$, we conclude that the right hand side of \eqref{middle-aa} is greater than 
\[
(-1/5+\sqrt{\eta})n+(4/5-\tfrac45\sqrt{\eta})n-1\ge (3/5+3\eta)n
\]
when $\eta \le 1/175^2$.
Hence $\card{(A-A)_{+}} \ge \left(3/5+3\eta\right)n$, as promised.
\end{proof}

Claim \ref{claim:middle-aa} obviously contradicts \eqref{eq:middle-aa-upper-bound}. This finishes our proof of Lemma \ref{lem:middle}.
\end{proof}

\subsection{Small range}\label{subsec:structure-small}
This section is devoted to the proof of Lemma \ref{lem:small}. As the proof is quite complicated, we first give a high level overview of our approach. Let $A_0=A\cap [n/2]$. The proof naturally splits into four steps
\begin{itemize}
	\item[1.] Show that $\card{A_0} \ge (1/5-o(1))n$ using Lemma \ref{lem:bootstrap} (i). This step is performed in Claim \ref{claim:small-balanced}.
	\item[2.] Use the estimate from the first step together with inverse theorems (Lemmas \ref{thm:Lev-Smeliansky} and \ref{thm:Jin}) to show that $A_0 \subseteq I_a\cup I_b$, where $I_a=\{a,a+d,\ldots,a+(\el_a-1)d\}$, $I_b=\{b,b+d,\ldots,b+(\el_b-1)d\}$, and $\el_a+\el_b=(1+o(1))\card{A_0}$. This is performed in Claim \ref{claim:small-structure-I}.
	\item[3.] Show that $A_0$ is contained in either $F_{1,4}$ or $F_{2,3}$ (Claim \ref{claim:small-a0-structure-II}). This step is performed as follows:
	\begin{itemize}
	\item[3.1] Combining steps $1$ and $2$ and the property that $A_0$ is sum-free, we obtain a number of inequalities that must be satisfied by the endpoints of $I_a$ and $I_b$.
	\item[3.2] Use the inequalities from the previous step to show that $d=5$, and either $\{a,b\}\equiv \{1,4\} \Mod{5}$ or $\{a,b\}\equiv \{2,3\} \Mod{5}$.
	\end{itemize}
	\item[4.] We use a `bootstrapping' argument (Lemma \ref{lem:bootstrap}) to upgrade the `$50\%$-structured characterisation' of $A$ from step $3$ to a $100\%$-structured characterisation.
\end{itemize}

Our bootstrapping lemma is the following simple result, proven in the appendix, which states that if a set $A$ of integers is dense in some interval $I$, then the difference set and sumset of $A$ contain long subintervals of $I-I$ and $I+I$ respectively.

\begin{lemma}[Folklore]\label{lem:bootstrap}
	Every finite set $A$ of integers has the following properties:
	\begin{itemize}
		\item[(i)] $A-A$ contains $[2\card{A}-\el(A)-1]$;
		\item[(ii)] If $A\subseteq [0,k]$ for some positive $k$, then $2A$ contains $[2k-2\card{A}+2,2\card{A}-2]$.
	\end{itemize}
\end{lemma}
These properties are only useful when the size of $A$ is at least $\el(A)/2+1$, though it is convenient not to make this a requirement.

\begin{proof}[Proof of Lemma \ref{lem:small}]
Throughout the proof, let $A_0=A\cap [n/2], A_1=A\setminus A_0$, and $m_e=\min(A\cap E)$.
We shall use Lemma \ref{lem:bootstrap} to show that $\card{A_0}$ is relatively large.

\begin{claim}\label{claim:small-balanced}
$\card{A_0} \ge (1/5-38\sqrt{\eta})n$.
\end{claim} 
\begin{proof}
To obtain a contradiction, suppose $\card{A_0} < \left(1/5-38\sqrt{\eta}\right)n$. As $\card{A}\ge \left(2/5-\eta\right)n$ by the assumption, this implies $\card{A_1}>(1/5+38\sqrt{\eta}-\eta)n$.  
We shall divide the proof into two cases, depending on whether  $d(A_1)=1$ or $d(A_1)>1$.
	
\noindent {\bf Case 1:}	$d(A_1)>1$.
	
We must have $d(A_1)\le 2$, since otherwise $\card{A_1} \le n/6+1<(1/5+38\sqrt{\eta}-\eta)n$, a contradiction. Thus $d(A_1)=2$, that is, either $A_1 \subseteq E$ or $A_1\subseteq O$. In either case, Lemma \ref{lem:bootstrap} (i) shows that $A_1-A_1$ contains all the even numbers between $0$ and $4\card{A_1}-n/2 \ge 3n/10$, giving $m_e \ge 3n/10$. 
	
We first consider the case $A_1 \subseteq E$. As $m_e \ge 3n/10$, we have $\card{A_0\cap [n/4]} \le n/8$, giving $\card{A_0\cap (n/4,n/2]} \ge \card{A_0}-n/8$. Thus $2\cdot A_0\cap (n/2,n]$ contains at least $\card{A_0}-n/8$ even numbers in $(n/2,n]\setminus A_1$. (Note that $2\cdot A_0\cap A_1=\emptyset$ since $A$ is sum-free.) It follows that the number of even integers in $(n/2,n]$ is at least
\[
\card{A_1}+\card{A_0}-n/8 \ge \left(2/5-\eta\right)n-n/8>21n/80
\]
for $\eta$ small, which is impossible.
	
We are left with the case $A_1 \subseteq O$. Let $M_e=\max(A\cap E)$, and let $O'$ denote the set of all the odd numbers less than $M_e$ in $A$. We have already shown that $m_e\ge 3n/10$. In addition, since $A_1 \subseteq O$, we  have $n/2 \ge M_e$. As $A$ is sum-free, $O'+\{M_e\}$ is a subset of $ \{M_e+1,M_e+3,\ldots,2M_e-1\}\setminus A$, and so $A$ has at most $M_e/2-\card{O'}$ odd elements in $[M_e,2M_e]$. Moreover, $(2M_e,n]$ contains at most $(n-2M_e)/2$ odd numbers. Thus $\card{A\cap O} \le \card{O'}+(M_e/2-\card{O'})+(n-2M_e)/2=(n-M_e)/2$, and hence $\card{A\cap E}=\card{A}-\card{A\cap O}\ge (2/5-\eta)n-(n-M_e)/2 \ge M_e/2-n/8$ when $\eta$ is small enough. However, $\card{A\cap E} \le (M_e-m_e)/2+1 \le M_e/2-3n/20+1$ since $m_e\ge 3n/20$. Using these bounds yields $3n/20-1 \le n/8$, which is impossible for $n$ large.
	
\noindent {\bf Case 2:}	$d(A_1)=1$.

Due to Lemma \ref{lem:bootstrap} (i), we have $A_1-A_1 \supseteq [2\card{A_1}-\el(A_1)-1]$. As $A$ is sum-free, it follows that $2\card{A_1}-\el(A_1)-1<\min(A) \le 35\sqrt{\eta}n$, giving $\el(A_1) \ge 2\card{A_1}-35\sqrt{\eta}n-1$.
Since $d(A_1)=1$, Lemma \ref{lem:Lev-Smeliansky} implies
\[
\card{(A_1-A_1)_{+}} \ge \min\left\{\tfrac12(\card{A_1}+\el(A_1)-2),\tfrac32\card{A_1}-2\right\} \ge \tfrac32\card{A_1}-18\sqrt{\eta}n
\]
for $1/n\le \eta \le c$. Moreover, since $A_0$ and $(A_1-A_1)_{+}$ are disjoint subsets of $[n/2]$, we see that $n/2 \ge \card{A_0}+\card{(A_1-A_1)_{+}}$. From these estimates we obtain
\[
n/2 \ge \card{A_0}+3\card{A_1}/2-18\sqrt{\eta}n \ge (2/5-\eta)n+\card{A_1}/2-18\sqrt{\eta}n.
\] 
So $\card{A_1} \le (1/5+2\eta+36\sqrt{\eta})n < (1/5+38\sqrt{\eta}-\eta)n$ for small $\eta$, a contradiction.
\end{proof}

In the rest of the proof, we use the $\ka$-notation for constants tending to zero as their parameters do so, that is, $\ka(\eta) \rightarrow 0$ whenever $\eta \rightarrow 0$.

We shall infer from Claim \ref{claim:small-balanced} that $\card{A_0+A_0}/\card{A_0}$ is small, and then rely on the inverse theorems (Lemmas \ref{thm:Lev-Smeliansky} and \ref{thm:Jin}) to get detailed structural information on $A_0$.

\begin{claim}\label{claim:small-structure-I}
The set $A_0$ has the following properties:
\begin{itemize}
\item[(i)] $d(A_0)=1$;
\item[(ii)] $A_0 \subseteq P_1\cup P_2$ for some arithmetic progressions $P_1$ and $P_2$ with the same step and $\card{P_1}+\card{P_2} \le (1+\ka(\eta))\card{A_0}$.
\end{itemize}
\end{claim}
\begin{proof}
(i) Toward a contradiction, suppose $d(A_0)>1$. We must have $d(A_0)<3$ because $\card{A_0}\ge (1/5-\ka(\eta))n>n/6+1$ by Claim \ref{claim:small-balanced}. Hence $A_0\subseteq E$ or $A_0\subseteq O$. If $A_0 \subseteq E$, then $\{a/2:a\in A_0\}$ is not sum-free since it is a set in $[n/4]$ of size $\card{A_0}>n/8+1$, contradicting our assumption that $A_0$ is sum-free. Now suppose $A_0\subseteq O$. Then $m_e>n/2$. To bound $\card{A}$, we partition $A=A'\dcup A''$, in which $A'=A\cap [m_e-1]$ and $A''=A\cap [m_e,n]$. Since $m_e=\min(A\cap E)$ and $A$ is sum-free, $A'$ and $m_e-A'$ are disjoint sets of odd numbers in $[m_e-1]$, giving $\card{A'} \le m_e/4$. To deal with $A''$, we note that $A''-A''$ contains $[2\card{A''}-(n-m_e)-2]$ by Lemma \ref{lem:bootstrap} (i). As $A$ is sum-free, it follows that $2\card{A''}-(n-m_e)-2 \le \min(A)=\ka(\eta)n$, resulting in $\card{A''} \le (n-m_e)/2+\ka(\eta)n$. Therefore, we have $\card{A}=\card{A'}+\card{A''} \le n/2-m_e/4+18\sqrt{\eta}n \le (1/4+\ka(\eta))n$, contradicting the assumption that $\card{A} \ge (2/5-\eta)n$.

\noindent (ii) Since $A$ is sum-free, $2A_0$ and $A$ are disjoint subsets of $[n]$. Hence 
\[
\card{2A_0} \le n-\card{A} \le (3/5+\eta)n \le (3+\ka(\eta))\card{A_0}
\] 
as $\card{A}\ge (2/5-\eta)n$ by the assumption, and $\card{A_0}\ge (1/5-\ka(\eta))n$ due to Claim \ref{claim:small-balanced}. By applying Lemma \ref{thm:Lev-Smeliansky} when $\card{2A_0}\le 3\card{A_0}-4$ and Lemma \ref{thm:Jin} in the case $\card{2A_0} \ge 3\card{A_0}-3$, we deduce that $A_0$ satisfies one of the following conditions:
\begin{itemize}
	\item[(a)] $A_0$ is a subset of an arithmetic progression of length $(2+\ka(\eta))\card{A_0}$;
	\item[(b)] $A_0 \subseteq P_1\cup P_2$ for some arithmetic progressions $P_1$ and $P_2$ with the same step and $\card{P_1}+\card{P_2}\le (1+\ka(\eta))\card{A_0}$.
\end{itemize}
To prove property (ii), it thus suffices to show that case (a) is impossible. In this case $A_0$ is located in an interval of length $(2+\ka(\eta))\card{A_0}$, as $d(A_0)=1$ by property (i). Since $\min(A_0)=\ka(\eta)n$ by the assumption and $\card{A_0}\ge (1/5-\ka(\eta))n$ by Claim \ref{claim:small-balanced}, it follows that $\min(A_0)=\ka(\eta)\card{A_0}$ and $A_0 \subseteq [(2+\ka(\eta))\card{A_0}]$. By Theorem \ref{thm:Freiman-large-sum-free}, we thus have $A_0 \subseteq O$, which contradicts property (i).
\end{proof}

We shall use the previous claims to obtain the following characterisation of $A_0$.

\begin{claim}\label{claim:small-a0-structure-II}
Either $A_0 \subseteq F_{1,4}$ or $A_0\subseteq F_{2,3}$.
\end{claim}

Before we proceed with the proof of Claim \ref{claim:small-a0-structure-II}, we show how it implies the lemma. From Claim \ref{claim:small-a0-structure-II} we have $A_0 \subseteq F_{1,4}$ or $A_0\subseteq F_{2,3}$.
We shall show that if $A_0 \subseteq F_{1,4}$ then $A\subseteq F_{1,4}$. Conversely, suppose that there exists $a \in A\setminus F_{1,4}$. Since $a \not \equiv 1,4 \Mod{5}$, we can find $i,j \in \{1,4\}$ such that $a\equiv i+j \Mod 5$. Now Claim \ref{claim:small-balanced} tells us that $\card{A_0} \ge (1/5-\ka(\eta))n$. Let $X= A_0\cap (5\cdot \bZ+1)$ and $Y=A_0\cap (5\cdot \bZ+4)$. From Lemma \ref{lem:bootstrap} (ii), we learn that $2X$, $X+Y$ and $2Y$ contain all the elements in $[\ka(\eta)n,(1-\ka(\eta))n]$ of $5\cdot\bZ+2$, $5\cdot \bZ$ and $5\cdot\bZ+3$, respectively. As $2A\neq A$ and $\card{A}\ge (2/5-\eta)n$, this shows that $a \ge (1-\ka(\eta))n$, and that $A$ contains all but at most $k(\eta)n$ elements of $[n]\cap (\{1,4\}+5\cdot \bZ)$.  Hence both $A$ and $a-A$ contain all but at most $\ka(\eta)n$ elements of $[n]\cap (5\cdot\bZ+ j)$, and so $A\cap (a-A)\ne \emptyset$, 
contradicting the assumption that $A\cap (A-A)=\emptyset$. 
In much the same way, the condition $A_0\subseteq F_{2,3}$ would force $A\subseteq F_{2,3}$.
\end{proof}

We close this section by deducing Claim \ref{claim:small-a0-structure-II} from Claims \ref{claim:small-balanced} and \ref{claim:small-structure-I}.

\begin{proof}[Proof of Claim \ref{claim:small-a0-structure-II}]
Finally we come to what is, in some sense, the trickiest part of our proof. Due to Claim \ref{claim:small-structure-I} (ii), there exist two arithmetic progressions $I_a=\{a,a+d,\ldots,a+(\el_a-1)d\}$ and $I_b=\{b,b+d,\ldots,b+(\el_b-1)d\}$ in $[n/2]$ such that $A_0\subseteq I_a\cup I_b$ and $\card{I_a}+\card{I_b} \le (1+\ka(\eta))\card{A_0}$. In particular,
	\begin{equation}\label{eq:small-APs}
	\card{A_0\cap I_u} \ge \card{I_u}-\ka(\eta)n \ \text{for every} \ u \in \{a,b\}.
	\end{equation}
	Clearly $\card{A_0} \le \card{I_a}+\card{I_b} \le n/d+2$. Combined with the bound $\card{A_0} \ge (1/5-\ka(\eta))n$ from Claim \ref{claim:small-balanced} we get $d\le 5$. We distinguish three cases $d=1, d=2$ and $d\in \{3,4,5\}$.
	
	\noindent {\bf Case 1:}	$d=1$.
	
	In this case both $I_a$ and $I_b$ are intervals. Without loss of generality we can assume that $\card{I_a} \ge \card{I_b}$. Since $A_0\subseteq I_a\cup I_b$, it follows that $\card{I_a}\ge \card{A_0}/2$, and so from \eqref{eq:small-APs} and Lemma \ref{lem:bootstrap} (i) we have $A_0-A_0 \supseteq (A_0\cap I_a)-(A_0\cap I_a)\supseteq [(1/2-\ka(\eta))\card{A_0}]$. But $\min(A_0)=\ka(\eta)\card{A_0}$ by the assumption, resulting in $ A_0\cap (A_0-A_0) \ne \emptyset$, a contradiction.
	
	\noindent {\bf Case 2:} $d=2$.
	
Since $d(A_0)=1$ by Claim \ref{claim:small-structure-I} (i), we must have $A_0\cap I_a \ne \emptyset, A_0\cap I_b \ne \emptyset$ and $a\ne b \Mod{2}$. So we can assume that $a \equiv 0 \Mod{2}$ and $b \equiv 1 \Mod{2}$. For $u\in \{a,b\}$, denote by $m_u$ the smallest element of $A\cap I_u$. We have $\min\{m_a,m_b\}=\min(A)=\ka(\eta)n$ by the assumption. We thus have $m_a \le n/30$, or $m_a>n/30$ and $m_b=\ka(\eta)n$.
	
	We first deal with the case $m_a \le n/30$. We claim that $\card{I_u} \le n/20$ for all $u\in \{a,b\}$. If this is not true then $\card{I_u} \ge n/20$ for some $u\in \{a,b\}$. From \eqref{eq:small-APs} and Lemma \ref{lem:bootstrap} (i), it follows that $(A_0\cap I_u)-(A_0\cap I_u)$ contains all the even numbers between $1$ and $(1-\ka(\eta))n/20$. Since $m_a \le n/30$, this leads to $m_a\in A_0-A_0$, a contradiction. We thus have $\card{A_0}\le \card{I_a}+\card{I_b} \le n/10$, contradicting Claim \ref{claim:small-balanced}.
	
	We now consider the case $m_a>n/30$ and $m_b=\ka(\eta)n$. For $u\in \{a,b\}$, let $M_u$ be the largest element of $I_u$. As $I_a\subseteq [n/2]$, we have the constraint $C1:M_a \le n/2$. We next show that $m_a$ and $M_a$ satisfy  $C2:M_a \le 2 m_a+n/20$. Indeed if $M_a \ge 2 m_a+n/20$, then we can deduce from \eqref{eq:small-APs} that $A_0\cap I_a$ and $m_a+(A_0\cap I_a)$ would have at least $(1-\ka(\eta))n/40$ even elements in common, which contradicts the assumption that $A$ is sum-free. We shall need one more constraint $C3:m_a\ge (2-\ka(\eta))M_b$. Indeed as $A_0\cap I_a$ is a sum-free subset of even integers in $[n/2]$, we find $\card{A_0\cap I_a} \le n/4+1$. Together with the estimate $\card{A_0}\ge (1/5-\ka(\eta))n$ from Claim \ref{claim:small-balanced}, we see that $\card{A_0\cap I_b} \ge (1/20-\ka(\eta))n$. From \eqref{eq:small-APs} and Lemma \ref{lem:bootstrap} (ii), it follows that $2(A_0\cap I_b)$ contains all the even numbers of $[\ka(\eta)n,(2-\ka(\eta))M_b]$. As $A_0$ is sum-free and $m_a>n/30$, this implies $m_a\ge (2-\ka(\eta))M_b$, as claimed. Under the constraints $C1$, $C2$ and $C3$, one has
	\begin{align*}
	\card{A_0} =\card{A_0\cap I_a}+\card{A_0\cap I_b}& \le \tfrac12(M_a-m_a+2)+\tfrac12 (M_b+1)\\
	&=\tfrac38M_a+\tfrac18(M_a-2m_a)+\tfrac14(2M_b-m_a)+\tfrac32 \le (\tfrac{31}{160}+\ka(\eta))n,	
	\end{align*}
	which contradicts the lower bound $\card{A_0}\ge (1/5-\ka(\eta))n$ from Claim \ref{claim:small-balanced}.
	
	\noindent {\bf Case 3:}	$d\in \{3,4,5\}$.
	
For each $u\in \{a,b\}$, let $\al_u$ and $\be_u$ be two real numbers such that $\min(I_u)=\al_u n/2$ and $\max(I_u)=\be_u n/2$. Set $\eps=1/1000$. We first show that for $\eta>0$ sufficiently small, the parameters $\al_a,\beta_a,\al_b$ and $\be_b$ satisfy the following constraints:
	\begin{itemize}
		\item[$(C1)$] $0 \le \al_u \le \be_u \le 1$ for each $u\in \{a,b\}$,
		\item[$(C2)$] $(\be_a-\al_a)+(\be_b-\al_b) \ge 2d/5-\eps$,
		\item[$(C3)$] If $u,v,w \in \{a,b\}$ and $u+v \equiv w \Mod{d}$, then $\be_u+\be_v \le \al_w+\eps$ or $\be_w \le \al_u+\al_v+\eps$.   
	\end{itemize}
	Indeed, as $I_a$ and $I_b$ are subsets of $[n/2]$ the first constraint follows. The second holds since $\card{I_a}+\card{I_b} \ge \card{A_0} \ge (1/5-\ka(\eta))n$ by Claim \ref{claim:small-balanced}. For the third, note first that from $(C1)$ and $(C2)$ one has 
	\begin{equation}\label{eq:width}
	\min\{\be_a-\al_a,\be_b-\al_b\}\ge 2d/5-\eps-1\ge 0.19.\footnote{In contrast, one may have $\min\{\be_a,\be_b\}=o(1)$ when $d=2$. This subtle difference between the two cases $d=2$ and $d\ge 3$ has forced us to treat them separately.}
	\end{equation}
	From \eqref{eq:small-APs} and \eqref{eq:width} we find $\min\{|A_0\cap I_a|,|A_0\cap I_b|\}\ge 0.19n/(2d)-\kappa(\eta)n>0.01n$. In particular, $A_0\cap I_u$ and $A_0\cap I_v$ are non-empty, and so \eqref{Cauchy-Davenport} implies
	\begin{align*}
	\card{(A_0\cap I_u)+(A_0\cap I_v)} &\ge \card{A_0\cap I_u}+\card{A_0\cap I_v}-1\\
	(\text{by \eqref{eq:small-APs}}) \quad & \ge \card{I_u}+\card{I_v}-1-\kappa(\eta)n=\card{I_u+I_v}-\kappa(\eta)n,	
    \end{align*}
    where the last equality holds because $I_u$ and $I_v$ are two arithmetic progressions with the same step. It follows that $(A_0\cap I_u)+(A_0\cap I_v)$ contains all but at most $\kappa(\eta)n$ elements of the arithmetic progression $\{x|x\in [(\al_u+\al_v)\frac{n}{2},(\be_u+\be_v)\frac{n}{2}], x\equiv u+v \Mod{d}\}$. On the other hand, \eqref{eq:small-APs} tells us that $A_0\cap I_w$ contains all but at most $\kappa(\eta)n$ members of the arithmetic progression $\{x|x\in [\al_w n/2,\be_w n/2], x\equiv w \equiv u+v \Mod{d}\}$. Therefore we must have $\be_u+\be_v \le \al_w+\eps$ or $\be_w \le \al_u+\al_v+\eps$, as otherwise $(A_0\cap I_u)+(A_0\cap I_v)$ and $A_0\cap I_w$ would have at least $\min\{\eps \cdot \frac{n}{2d}-\kappa(\eta)n,0.01n\}>0$ elements in common, which contradicts the assumption that $A_0$ is sum-free.  
	
	In what follows we shall exploit the constraints $(C1)$--$(C3)$ to show that the set $\{a,b\}$ is sum-free modulo $d$.
	Note that since $d(A_0)=1$, we must have $a\ne b \Mod{d'}$ for every divisor $d'$ of $d$ with $d'>1$. Due to symmetry between $a$ and $b$, we thus only need to take care of the following three cases.
	
	{\bf Case 3.1:} $d=3, a\equiv 1 \Mod{3}$ and $b\equiv 2 \Mod{3}$.
	Using $(C3)$ with $u=v=a$ and $w=b$, we deduce that $2\be_a \le \al_b+\eps$ or $\be_b \le 2\al_a+\eps$. If $2\be_a \le \al_b+\eps$, then
	\[
	(\be_a-\al_a)+(\be_b-\al_b)=\be_b+(2\be_a-\al_b)-(\be_a+\al_a) \le 1+\eps <2d/5-\eps
	\]
	since $\be_b \le 1$ and $\al_a,\be_a \ge 0$ by $(C1)$, which contradicts $(C2)$. We thus have $\be_b \le 2\al_a+\eps$. By  symmetry we also get $\be_a \le 2\al_b+\eps$. Hence 
	\[
	(\be_a-\al_a)+(\be_b-\al_b)=\tfrac12(\be_a+\be_b)+\tfrac12(\be_a-2\al_b)+\tfrac12(\be_b-2\al_a) \le 1+\eps
	\]
	since $\be_a,\be_b \le 1$ by $(C1)$. But this bound is inconsistent with $(C2)$.
	
	{\bf Case 3.2:} $a\equiv 0 \Mod{d}$. Property $(C3)$ tells us that $\be_a+\be_b \le \al_b+\eps$ or $\be_b \le \al_a+\al_b+\eps$. If the former condition occurs, then from $(C1)$ and $(C2)$ we get 
	\[
	1+\eps\ge \al_b+\eps\ge \be_a+\be_b\ge (\be_a-\al_a)+(\be_b-\al_b) \ge 2d/5-\eps,
	\]
	which is impossible. Hence $\be_b \le \al_a+\al_b+\eps$. Combined with the constraint $\be_a \le 1$ from $(C1)$, we again get a contradiction  
	\[
	(\be_a-\al_a)+(\be_b-\al_b)=\be_a+(\be_b-\al_a-\al_b) \le 1+\eps<2d/5-\eps.
	\]
	{\bf Case 3.3:} $d\in \{4,5\}$, $a,b\not \equiv 0 \Mod{d}$, and $a\not \equiv b \Mod{d'}$ for every divisor $d'$ of $d$ with $d'>1$. We begin by reducing to the case that $\{a,b\}$ is sum-free modulo $d$. Indeed consider the relation $b\equiv 2a \Mod{d}$. As in the proof of Case 3.1, this would imply $\be_b \le 2\al_a+\eps$. Thus
	\[
	(\be_a-\al_a)+(\be_b-\al_b)=(\be_a+\tfrac12\be_b)+\tfrac12(\be_b-2\al_a)-\al_b \le 3/2+\eps/2,
	\]
	since $\be_a,\be_b \le 1$ and $\al_b \ge 0$ by $(C1)$. But once again this contradicts $(C2)$. The case $a\equiv 2b \Mod{d}$ follows by symmetry.
	
	We have shown that the set $\{a,b\}$ is sum-free modulo $d$. Combined this with the condition that $a\not \equiv b\Mod{d'}$ for every divisor $d'$ of $d$ with $d'>1$, we conclude that, up to a permutation of $a$ and $b$, either $a\equiv 1 \Mod{5}$ and $b\equiv 4\Mod{5}$ or $a\equiv 2 \Mod{5}$ and $b\equiv 3\Mod{5}$.     
\end{proof}

\section{The number of $2$-wise sum-free sets}
\label{sec:counting}  

In this section, we prove Theorem \ref{thm:counting}.

\subsection{Proof overview}\label{subsec:counting-overview}

Recently the method of {\bf containers} has emerged as a powerful tool for tackling various problems in combinatorics. Roughly speaking this method states that the independent sets in many `natural' hypergraphs exhibit a certain kind of `clustering', which allows one to count them one cluster at a time. Balogh, Morris and Samotij \cite{BMS} and Saxton and Thomason \cite{ST}, proved general container theorems for hypergraphs $\mathcal{H}$ whose edges are fairly `evenly distributed' over the vertices of $\mathcal{H}$.

In the proof of Theorem \ref{thm:counting}, we shall apply a special case of a container result of Hancock, Staden and Treglown \cite[Theorem 4.7]{HST17}. We remark that their proof uses the theorems of Balogh, Morris and Samotij \cite{BMS}, and Saxton and Thomason \cite{ST}.

\begin{lemma}[Hancock--Staden--Treglown]\label{lem:containers}
There exists a collection $\cC$ of subsets of $[n]^2$ with the following three properties:
\begin{itemize}
	\item[(i)] If $(A_1,A_2)$ is a pair of disjoint sum-free subsets of $[n]$, then there exists a pair $(C_1,C_2)\in \cC$ such that $(A_1,A_2)\subseteq (C_1,C_2)$;
	\item[(ii)] $\card{\cC}=2^{o(n)}$;
	\item[(iii)] For any $(C_1,C_2) \in \cC$, each $C_i$ contains at most $o(n^2)$ Schur triples.
\end{itemize} 	
\end{lemma}

We refer to the elements of $\cC$ from Lemma \ref{lem:containers} as {\bf containers}.  

{\bf A counting strategy.} Our general strategy is influenced by the approach used in \cite{BLST17}, which in turn dates back to earlier s of Cameron and Erd\H{o}s \cite{Cameron-Erdos} and Green \cite{Gr05}. Given $A\in \SF_2(n)$ and a partition $A=A_1\dcup A_2$ of $A$ into two sum-free sets, we consider some container $(C_1,C_2) \in \cC$ with $(A_1,A_2) \subseteq (C_1,C_2)$. As $\cC$ is so small, the number of $A$ for which $\card{C_1\cup C_2} \le (4/5-\eta)n$ is $o(2^{4n/5})$. If, however, $\card{C_1\cup C_2}\ge (4/5-\eta)n$ then it is possible to say something about the structure of $(C_1,C_2)$, and hence about the structure of a typical set $A\in \SF_2(n)$. We then use a direct argument rather than counting such sets within the containers.

As discussed above, we need to get a handle on the structure of large containers. For this purpose, we first deduce from Theorem \ref{thm:structure} a structural result on $2$-wise sum-free sets of size close to $4n/5$, which may be of independent interest. 

\begin{prop}\label{prop:stability}
There exists an absolute positive constant $c$ such that the following holds for every $n\in \bN$ and every $\eta \in \bR$ with $2/n \le \eta \le c$. Let $C_1$ and $C_2$ be two sum-free sets (not necessarily disjoint) in $[n]$ with $\card{C_1\cup C_2} \ge (4/5-\eta)n$. Then, up to a permutation of $C_1$ and $C_2$, one of the following situations occurs:
\begin{itemize}
\item[(i)] $\card{C_1 \setminus F_{1,4}}+\card{C_2\setminus F_{2,3}} \le 14\eta n$;
\item[(ii)] $\card{C_1\setminus I_1}+\card{C_2\setminus I_2} \le 2424\sqrt{\eta}n$, where $I_1=\left(\frac{n}{5},\frac{2n}{5}\right] \cup \left(\frac{4n}{5},n\right]$ and $I_2=\left(\frac{2n}{5},\frac{4n}{5}\right]$.
\end{itemize}
\end{prop}

\begin{proof}[Proof of Proposition \ref{prop:stability}]
We begin by showing that neither $\card{C_1}$ nor $\card{C_2}$ are substantially greater than $2n/5$. 
\begin{claim}\label{claim:stability}
$\max\{\card{C_1},\card{C_2}\} \le (2/5+3\eta)n$.
\end{claim}
\begin{proof}
Denote $\t{C_2}=C_2\setminus C_1$ and $R=[n]\setminus (C_1\cup C_2)$. As $\card{C_1\cup C_2} \ge (4/5-\eta)n$, one has $\card{R} \le (1/5+\eta)n$. Let $c_1,\ldots,c_k$ be the elements of $C_1$, indexed in increasing order, and let $D=\{c_2-c_1,\ldots,c_k-c_1\}$. Since $C_1$ is sum-free, $D\cap C_1=\emptyset$, and consequently $D \subseteq \t{C_2}\cup R$. It follows that
\[
|D\cap \t{C_2}| \ge \card{D}-\card{R}.
\]

Let $\el=|D\cap \t{C_2}|$. From the definition of $D$, there are $\el$ distinct numbers $i_1, \ldots, i_{\el}$ in $\{2,3,\ldots,k\}$, indexed in increasing order, so that $D\cap \t{C_2}= \{c_{i_1}-c_1,\ldots,c_{i_{\el}}-c_1\}$.
Since $\t{C_2}$ is sum-free, one has $c_{i_t}-c_{i_1}=(c_{i_t}-c_1)-(c_{i_1}-c_1) \notin \t{C_2}$ for all $t$ with $2 \le t \le \el$. Also $c_{i_t}-c_{i_1} \notin C_1$, as $C_1$ is sum-free. Hence $c_{i_t}-c_{i_1}\in R$ for each $t \in \{2,\ldots,\el\}$, and so $\card{R} \ge \el-1$. Thus
\[
|D\cap \t{C_2}|\le \card{R}+1.
\]
Using these bounds on $|D\cap \t{C_2}|$ gives $\card{C_1}=\card{D}+1 \le 2\card{R}+2 \le (2/5+3\eta)n$ when $\eta \ge 1/n$. In the same manner we can show $\card{C_2} \le (2/5+3\eta)n$.
\end{proof}

We consider the sets $\t{C_1}=C_1\setminus C_2$ and $\t{C_2}=C_2\setminus C_1$. Clearly one has
$\t{C_1}\cap\t{C_2}=\emptyset$. Since $\max\{\card{C_1},\card{C_2}\}\le (2/5+3\eta)n$ by Claim \ref{claim:stability} and $\card{C_1\cup C_2} \ge (4/5-\eta)n$ by the assumption, we find $\card{C_1\cap C_2} \le 7\eta n$  
and $\min\{|\t{C_1}|,|\t{C_2}|\} \ge (2/5-4\eta)n$. 
We shall derive the lemma from this information and Theorem \ref{thm:structure}. 

Applying Theorem \ref{thm:structure} to $\t{C_1}$ and $\t{C_2}$, and noting that $\min\{|\t{C_1}|,|\t{C_2}|\} \ge (2/5-4\eta)n$ and $\t{C_1}\cap\t{C_2}=\emptyset$, we conclude that, up to a permutation of $\t{C_1}$ and $\t{C_2}$, one of the following conditions must be true:
\begin{itemize}
	\item[(i')] $\t{C_1}\subseteq F_{1,4}$ and $\t{C_2}\subseteq F_{2,3}$;
	\item[(ii')] $\t{C_1} \subset \t{I_1}$ and $\min(\t{C_2}) \ge (2/5-4\eta)n$, where 
	\[
	\t{I_1}=\left[\left(\tfrac15-400\sqrt{\eta}\right)n,\left(\tfrac25+400\sqrt{\eta}\right)n\right] \cup \left[\left(\tfrac45-400\sqrt{\eta}\right)n,n\right].\]
\end{itemize}
If condition (i') holds, then $\card{C_1\setminus F_{1,4}}+\card{C_2\setminus F_{2,3}}\le 2\card{C_1\cap C_2} \le 14\eta n$. Suppose, then, that condition (ii') is true. In particular, one has $|\t{C_1}\setminus I_1| \le 1200 \sqrt{\eta}n+3$. Hence 
\[
\card{C_1\setminus I_1} \le |\t{C_1}\setminus I_1|+\card{C_1\cap C_2}\le (1200\sqrt{\eta}+7\eta)n+3,
\] 
as $\card{C_1\cap C_2} \le 7\eta n$. 

It remains to bound $\card{C_2\setminus I_2}$. From condition (ii') and the fact that $\t{C_1}\cap \t{C_2}=\emptyset$, we learn that $\t{C_2}\setminus I_2 \subseteq \t{C_2}\cap \t{I_1} \subseteq \t{I_1}\setminus \t{C_1}$. Thus $C_2\setminus I_2 \subseteq (\t{C_2}\setminus I_2)\cup (C_1\cap C_2) \subseteq (\t{I_1}\setminus \t{C_1}) \cup (C_1\cap C_2)$, leading to
\[
\card{C_2\setminus I_2} \le |\t{I_1}\setminus \t{C_1}|+\card{C_1\cap C_2}=|\t{I_1}|-|\t{C_1}|+ \card{C_1\cap C_2}\le (1200\sqrt{\eta}+11\eta)n+3,
\]
where the second inequality follows from condition (ii'), and in the last we evaluated $|\t{I_1}| \le (2/5+1200\sqrt{\eta})+3$, $|\t{C_1}| \ge (2/5-4\eta)n$ and $\card{C_1\cap C_2} \le 7\eta n$. From these upper bounds on $\card{C_1\setminus I_1}$ and $\card{C_2\setminus I_2}$, we find 
\[
\card{C_1\setminus I_1}+\card{C_2\setminus I_2} \le (2400\sqrt{\eta}+18\eta)n+6 \le 2424\sqrt{\eta}n.\qedhere
\]
\end{proof}

We also need a removal lemma of Green \cite[Corollary 1.6]{Green05} for  sum-free  sets.

\begin{lemma}[Green]\label{lem:Green-removal}
Suppose that $C\subseteq [n]$ is a set containing $o(n^2)$ Schur
triples. Then there exists a sum-free subset $\t{C}$ of $C$ such that $|C\setminus \t{C}|=o(n)$.
\end{lemma}

From Lemma \ref{lem:containers}, Proposition \ref{prop:stability} and Lemma \ref{lem:Green-removal} we obtain the following description of almost all $A\in \SF_2(n)$. Note that we shall identify each set $A\in \SF_2(n)$ with a pair $(A_1,A_2)$ of disjoint sum-free sets so that $A=A_1\dcup A_2$.

\begin{cor}\label{cor:small-containers}
Given $\de>0$, every set $A \in \SF_2(n)$, with at most $o(2^{4n/5})$ exceptions, has one of the following structures (up to a permutation of $A_1$ and $A_2$):
\begin{itemize}
	\item[(a)] $\card{A_1\setminus F_{1,4}}+\card{A_2\setminus F_{2,3}}\le \de n$;
	\item[(b)] $\card{A_1\setminus I_1}+\card{A_2\setminus I_2}\le \de n$, in which $I_1=\left(\frac{n}{5},\frac{2n}{5}\right]\cup \left(\frac{4n}{5},n\right]$ and $I_2=\left(\frac{2n}{5},\frac{4n}{5}\right]$.
\end{itemize}
\end{cor}

In the remainder of the paper we refer to sets that satisfy condition (a) and condition (b) from Corollary \ref{cor:small-containers} as {\bf type $(a)$} and {\bf type $(b)$} respectively. Note that Corollary \ref{cor:small-containers} implies that, in order to prove Theorem \ref{thm:counting}, it suffices to show that there are at most $O(2^{4n/5})$ sets $A\in \SF_2(n)$ of type (a) and type (b).

\begin{proof}[Proof of Corollary \ref{cor:small-containers}]
Let $\eta=\min\left\{\frac{\de}{29},\left(\frac{\de}{3430}\right)^2,\tfrac12c_{\ref{prop:stability}}\right\}$, where $c_{\ref{prop:stability}}$ is the absolute positive constant from Proposition \ref{prop:stability}. For each set $A\in \SF_2(n)$, we fix a pair $(A_1,A_2)$ of disjoint sum-free sets in $[n]$, and a container $(C_1,C_2)\in \cC$ such that $A=A_1\cup A_2$ and $(A_1,A_2)\subset (C_1,C_2)$. According to Lemma \ref{lem:containers} (ii), the number of set $A\in \SF_2(n)$ for which $\card{C_1\cup C_2} \le (4/5-\eta)n$ is certainly at most $2^{(4/5-\eta)n}\cdot 2^{o(n)}=o(2^{4n/5})$, so suppose  $\card{C_1\cup C_2} \ge (4/5-\eta)n$. By Lemma \ref{lem:Green-removal}, there exists a pair $(\t{C_1},\t{C_2})$ of sum-free sets such that  $(\t{C_1},\t{C_2})\subseteq (C_1,C_2)$ and $|C_1\setminus \t{C_1}|+|C_2\setminus \t{C_2}|=o(n)$. Observe that $|\t{C_1}\cup \t{C_2}|=\card{C_1\cup C_2}-o(n) \ge (4/5-2\eta)n$. Since $2\eta \le c_{\ref{prop:stability}}$ by the choice of $\eta$, we may appeal to Proposition \ref{prop:stability} to conclude that:
\begin{itemize}
\item[$(a')$] $|\t{C_1}\setminus F_{1,4}|+|\t{C_2}\setminus F_{2,3}|\le 28 \eta n$, or
\item[$(b')$] $|\t{C_1}\setminus I_1|+|\t{C_2}\setminus I_2|\le 3429 \sqrt{\eta} n$.
\end{itemize}
If case $(a')$ is true, then $\card{A_1\setminus F_{1,4}}+\card{A_2\setminus F_{2,3}}\le 28 \eta n +o(n)<\de n$ for $\eta \le \de/29$. If, however, case $(b')$ occurs then $\card{A_1\setminus I_1}+\card{A_2\setminus I_2} \le 3429 \sqrt{\eta} n+o(n)\le \de n$ since $\eta \le \left(\frac{\de}{3430}\right)^2$, completing the proof.
\end{proof}

\subsection{Restricted partitions and sumsets}\label{subsec:counting-lemmata}
In this section, we introduce some tools that are useful for counting
sets $A\in \SF_2(n)$ of type (a) and type (b).

A handy tool for the study of sumsets is Pl{\"u}nnecke's inequality \cite{Plunnecke}. 

\begin{lemma}[Pl{\"u}nnecke Inequality]
\label{lem:Plunnecke}
If $S$ is a set of integers and $\card{S+S} \le R\card{S}$, then 
\[
\card{kS} \le R^k\card{S}
\]
for any positive integer $k$.
\end{lemma}

We shall need the following bound on the number of $s$-subsets $S$ of $\{1,2,\ldots,D\}$ with $\card{S+S} \le R\card{S}$, due to Green and Morris \cite[Theorem 1.1]{GM16}.

\begin{lemma}[Green--Morris]\label{lem:Green-Morris}
	Fix $\de>0$ and $R>0$. Then the following holds for all integers $s$ with $s\ge s_0(\de,R)$. For any $D\in \bN$ there are at most
	\[
	2^{\de s}\binom{\frac12 Rs}{s}D^{\floor{R+\de}}
	\]
	sets $S\subseteq [D]$ with $\card{S}=s$ and $\card{S+S}\le R\card{S}$. 
\end{lemma}

Lemma \ref{lem:Green-Morris} will be used in conjunction with some estimates on binomial coefficients, which we list here for future reference. It is well-known that for every integers $n$ and $k$ with $0 \le k \le n$ and every real $\al$ with $0\le \al \le 1/2$, we have 
\begin{equation}\label{entropy}
\binom{n}{k} \le 2^{H(k/n)n}, \ \text{and} \ \sum_{i \le \al n}\binom{n}{i} \le 2^{H(\al)n}, 
\end{equation}
where $H(x)=-x\log_2(x)-(1-x)\log_2(1-x)$ is the binary entropy function.

Another component in our argument is a crude bound on the number of restricted integer partitions (see \cite[Lemma 5.1]{ABMS14}). 

\begin{lemma}\label{lem:restricted-partitions}
	Given $k,\el \in \bN$, let $p^{*}_{\el}(k)$ denote the number of integer partitions of $k$ into $\el$ distinct parts. Then
	\[
	p^{*}_{\el}(k) \le \left(\frac{e^2k}{\el^2}\right)^{\el}.
	\]	
\end{lemma}

To handle sets with large sumset, we shall apply the following lower tail estimate, which is a special case of Janson's inequality (see \cite[Theorem 2.14]{JLR}).

\begin{lemma}[Janson Inequality]\label{lem:Janson-Inequality}
	Suppose that $\{U_i\}_{i\in I}$ is a collection of subsets of a finite set $\Ga$. Let
	\[
	\mu=\sum_{i \in I}\left(\tfrac12\right)^{\card{U_i}} \quad \text{and} \quad \De=\sum_{i\sim j} \left(\tfrac12\right)^{\card{U_i\cup U_j}}, 
	\]
	where the second sum is over ordered pairs $(i,j)$ such that $i\ne j$ and $U_i\cap U_j \ne \emptyset$.
	Then the number of subsets of $\Ga$ that contain at most $\mu/2$ sets $U_i$ is at most 
	\[
	e^{-\mu^2/(8\mu+8\De)}\cdot 2^{\card{\Ga}}.
	\]
\end{lemma}

\subsection{Counting sets of type $(a)$ and type $(b)$}\label{subsec:counting-containers}
Throughout we identify each set $A\in \SF_2(n)$ with a pair $(A_1,A_2)$ of disjoint sum-free sets so that $A=A_1\dcup A_2$.

The following lemma deals with sets of type $(a)$.

\begin{lemma}\label{lem:arithmetic-containers}
	There are $(1+o(1))2^{\ceil{4n/5}}$ sets $A\in \SF_2(n)$ of type $(a)$, provided that $\de>0$ is sufficiently small.
\end{lemma}
\begin{proof} 
	There are $2^{\ceil{4n/5}}$ sets $A\in \SF_2(n)$ with $A_1\subseteq F_{1,4}$ and $A_2\subseteq F_{2,3}$. So, to prove the lemma, it suffices to show that the number of type $(a)$ sets $A$ with $0<\card{A_1\setminus F_{1,4}}+\card{A_2\setminus F_{2,3}}\le \de n$ is $o(2^{4n/5})$. By symmetry we only need to deal with the case that $A_1\setminus F_{1,4}$ contains at least one element, $t$ say. If $t<n/2$, then we may select $n/20$ disjoint pairs $(x,x+t)$ in $F_{1,4}$, and $A_1$ can not contain both of the elements of any of them since it is sum-free. The number of choices for the pair $(A_1\cap F_{1,4},A_2\cap F_{2,3})$ is thus no more than $2^{3n/10}3^{n/20}\cdot 2^{2n/5}=2^{7n/10}3^{n/20}$. Furthermore, since $\card{A_1\setminus F_{1,4}}+\card{A_2\setminus F_{2,3}}\le \de n$, the number of pairs $(A_1\setminus F_{1,4},A_2\setminus F_{2,3})$ is at most $\left(\sum_{i \le \de n}\binom{n}{i}\right)^2 \le 2^{2H(\de)n}$, due to \eqref{entropy}. We deduce that there are at most $2^{7n/10}3^{n/20}\cdot 2^{2H(\de)n}=o(2^{4n/5})$
ways to choose $(A_1,A_2)$. If $t\ge n/2$ then a very similar argument applies with pairs $(x,x-t)$.   
\end{proof}

We now turn our attention to sets of type $(b)$. Note that  Corollary \ref{cor:small-containers}, Lemmas  \ref{lem:arithmetic-containers} and \ref{lem:continuous-containers} together imply Theorem \ref{thm:counting}.

\begin{lemma}\label{lem:continuous-containers}
If $\de>0$ is sufficiently small, then there are $O(2^{4n/5})$ sets $A\in \SF_2(n)$ of type $(b)$. 
\end{lemma}

The proof of Lemma \ref{lem:continuous-containers} is fairly long and technical so, in order to aid the reader, we shall start by giving a brief sketch. The argument is split into four claims; the first three being relatively straightforward, and the last being somewhat more involved.

We begin, in Claim \ref{claim:counting-reduction-I}, by using a direct argument to give a description of almost all sets $A\in \SF_2(n)$ of type $(b)$. In Claims \ref{claim:counting-reduction-II}, \ref{claim:counting-reduction-III} and \ref{claim:counting-reduction-IV}, we use this description to bound the number of sets $A\in \SF_2(n)$ with $S=A\cap [n/5]$ fixed. Specifically, writing $\el=\card{S}$ and $k=\sum_{a\in S}(n/5-a)$, in Claim \ref{claim:counting-reduction-II} we use Claim \ref{claim:counting-reduction-I}, Lemmas \ref{lem:restricted-partitions} and \ref{lem:Janson-Inequality} to deal with the case
$k\gg \el^2$. Then, in Claim \ref{claim:counting-reduction-III}, we use Claim \ref{claim:counting-reduction-I} and Lemma \ref{lem:restricted-partitions} to handle the case $k=O(\el^2)$ and $\card{S+S} \gg \card{S}$. Finally, in Claim \ref{claim:counting-reduction-IV}, we treat the remaining (hard) case; however, since we now have $\card{S+S}=O(\card{S})$, we may apply Lemma \ref{lem:Green-Morris} in place of Lemma \ref{lem:restricted-partitions}.

\begin{proof}[Proof of Lemma \ref{lem:continuous-containers}]
Fix $\de>0$ sufficiently small, and let $n\in \bN$. We shall show that there are at most $O(2^{4n/5})$ sets $A\in \SF_2(n)$ of type $(b)$. Since for us the residue of $n$ modulo $5$ will not matter, we assume for simplicity throughout the proof that $n$ is divisible by $5$. We begin by proving that a typical set $A\in \SF_2(n)$ of type $(b)$ has the following property:
\[
	(\al) \hspace{0.1cm} A_1 \subseteq \left[\left(\tfrac15-\tfrac{1}{100}\right)n,\left(\tfrac25+\tfrac{1}{100}\right)n \right] \cup \left[\left(\tfrac45-\tfrac{1}{100}\right)n,n \right] \hspace{0.1cm} \text{and} \hspace{0.1cm} A_2 \subseteq \left[\left(\tfrac25-\tfrac{1}{100}\right)n,\left(\tfrac45+\tfrac{1}{100}\right)n \right].
\]

\begin{claim}\label{claim:counting-reduction-I}
With $o(2^{4n/5})$ exceptions, all sets $A \in \SF_2(n)$ of type $(b)$ satisfy $(\al)$.
\end{claim}
\begin{proof}
Let $A\in\SF_2(n)$ be a set of type $(b)$ that does not posses property $(\al)$. If $A_1$ contains an element $t\in \left[(\tfrac15-\tfrac{1}{100})n\right] \cup \left[\left(\tfrac25+\tfrac{1}{100}\right)n,\left(\tfrac45-\tfrac{1}{100}\right)n\right]$, then we can pick at least $n/400$ disjoint pairs $(x,x+t)$ in $I_1$. Thus the number of ways to choose $(A_1\cap I_1,A_2\cap I_2)$ is at most $2^{79n/200}3^{n/400}\cdot 2^{2n/5}=2^{159n/200}3^{n/400}$. In addition, since $\card{A_1\setminus I_1}+\card{A_2\setminus I_2} \le \de n$, there are at most $2^{2H(\de)n}$ choices for $(A_1\setminus I_1,A_2\setminus I_2)$. From these estimates it follows that there are at most $2^{159n/200}3^{n/400}\cdot 2^{2H(\de)n}=o(2^{4n/5})$ possible assignments for $(A_1,A_2)$. The same conclusion can be drawn for the case that $A_2$ has at least one element in $\left[(\tfrac25-\tfrac{1}{100})n\right] \cup \left[\left(\tfrac45+\tfrac{1}{100}\right)n,n\right]$.
\end{proof}

From now on we may restrict our attention to those $A\in \SF_2(n)$ satisfying $(\al)$. Let  
\[
S(A)=\{x\in A:x \le n/5\}
\]
denote the collection of elements of $A$ which are at most $n/5$. We shall count the number of sets $A \in \SF_2(n)$ with $S(A)$ fixed. The following simple but crucial observation will be exploited several times to bound the number of ways to choose $A\cap \{n/5+1,\ldots,n\}$.

\begin{observation}
Every set $A\in \SF_2(n)$ with property $(\al)$ satisfies the following:
\begin{itemize}
\item[(i)] $S(A)=A_1\cap \left[\left(\tfrac15-\tfrac{1}{100}\right)n,\tfrac15 n \right]$, and $A\cap (S(A)+S(A)) \subseteq A_2\cap \left[\left(\tfrac25-\tfrac{1}{50}\right)n,\tfrac25 n \right]$;
\item[(ii)] If $X \subseteq A_2\cap \left[\left(\tfrac25-\tfrac{1}{50}\right)n,\tfrac25 n \right]$, then $A\cap \{4n/5+1,\ldots,n\}$ and $S+(A\cap 2X)$ are disjoint subsets of $\{4n/5+1,\ldots,n\}$.  	
\end{itemize}
\end{observation}
\begin{proof} To ease notation we shall write $S$ for $S(A)$. 
	
(i) The first statement holds since $A_2\cap [n/5]=\emptyset$ and $\min(A_1) \ge \left(\tfrac15-\tfrac{1}{100}\right)n$ by property $(\al)$. Since $S=A_1\cap \left[\left(\tfrac15-\tfrac{1}{100}\right)n,\tfrac15 n \right]$, we have $2S \subseteq 2A_1\cap \left[\left(\tfrac25-\tfrac{1}{50}\right)n,\tfrac25 n \right]$. As $A_1 \cap 2A_1=\emptyset$ and $A=A_1\cup A_2$, this forces 
\[
(A\cap 2S) \subseteq (A\cap 2A_1)\cap \left[\left(\tfrac25-\tfrac{1}{50}\right)n,\tfrac25 n \right] \subseteq A_2\cap \left[\left(\tfrac25-\tfrac{1}{50}\right)n,\tfrac25 n \right]. 
\]

(ii) As $X\subseteq A_2\cap  \left[\left(\tfrac25-\tfrac{1}{50}\right)n,\tfrac25 n \right]$, we have $2X\subseteq 2A_2 \cap  \left[\left(\tfrac45-\tfrac{1}{25}\right)n,\tfrac45 n \right]$. Since $A_2\cap 2A_2=\emptyset$ and $A=A_1\cup A_2$, it follows that 
\[
(A\cap 2X) \subseteq (A\cap 2A_2)\cap \left[\left(\tfrac45-\tfrac{1}{25}\right)n,\tfrac45 n \right] \subseteq A_1\cap \left[\left(\tfrac45-\tfrac{1}{25}\right)n,\tfrac45 n \right].
\]
As $S=A_1\cap \left[\left(\tfrac15-\tfrac{1}{100}\right)n,\tfrac15 n \right]$ due to (i), this implies $S+(A\cap 2X) \subseteq 2A_1 \cap \left[\left(1-\tfrac{1}{20}\right)n,n \right]$. In particular, one has $S+(A\cap 2X) \subseteq \{4n/5+1,\ldots,n\}$. Furthermore, the intersection of $A$ and $S+(A\cap 2X)$ is contained in 
$(A\cap 2A_1)\cap \left[\left(1-\tfrac{1}{20}\right)n,n \right] \subseteq A_2 \cap \left[\left(1-\tfrac{1}{20}\right)n,n \right]\overset{(\al)}{=}\emptyset$.
These properties imply the statement. 
\end{proof}

The remainder of the proof involves some careful counting using the observation as well as Lemmas \ref{lem:Green-Morris}, \ref{lem:restricted-partitions} and \ref{lem:Janson-Inequality}.
We shall break up the calculation into three claims. In the first two, we count the sets $A$ for which $\sum_{a\in S(A)}(n/5-a)$ is large (Claim \ref{claim:counting-reduction-II}), or $\sum_{a\in S(A)}(n/5-a)$ is small and $\card{S(A)+S(A)}$ is large (Claim \ref{claim:counting-reduction-III}). Finally we count the remaining sets in Claim \ref{claim:counting-reduction-IV}.

Let $\cS(k,\el)$ denote the collection of sets $S\subseteq [n/5]$ with $\card{S}=\el$ and 
\[
\sum_{a\in S}(n/5-a)=k.
\]
\begin{claim}\label{claim:counting-reduction-II}
For a given fixed $\ell \in \bN$, there are at most $e^{-\el}2^{4n/5}$ sets $A\in \SF_2(n)$ of type $(b)$ which satisfy $(\al)$ and with $S(A)\in \cS(k,\el)$ for some $k\ge \el^2/\de^2$.
\end{claim}
\begin{proof}
For $k\ge \el^2/\de^2$ and $S\in \cS(k,\el)$, let $\cI(S)$ denote the family of all sets $A\in \SF_2(n)$ of type $(b)$ that satisfy $(\al)$ and with $S(A)=S$. We shall first bound $\cI(S)$ and then sum over choices of $S$. Define the graph $G$ of `forbidden monochromatic pairs' by setting
\[
V(G)=\{n/5+1,\ldots, 2n/5\}, \ \text{and} \ E(G)=\{\{x,x+s\}:s\in S\}.
\]
We partition $\cI(S)=\cI_1(S)\dcup \cI_2(S)$, in which $\cI_1(S)$ consists of all those sets $A\in \cI(S)$ having the property that $A\cap V(G)$ contains at most $k/8$ edges of $G$. 

We shall use Janson Inequality to estimate $\card{\cI_1(S)}$.
Observe that $G$ has $k$ edges and maximum degree at most $2\el$, since $S(A)=S \in \cS(k,\el)$. Let $\mu$ and $\De$ be the quantities defined in the statement of Lemma \ref{lem:Janson-Inequality} and note that we are applying the lemma with $\card{\Ga}=n/5$. We have
\[
\mu=k\cdot \left(\tfrac12\right)^2=k/4 \quad \text{and} \quad \De \le 4k\el \cdot \left(\tfrac{1}{2}\right)^3 =k\el/2.
\]
Accordingly $\mu^2/(8\mu+8\De) \ge k/(96\el)$, and so the number of choices for $A\cap V(G)$ is at most $e^{-k/96\el} 2^{n/5}$. On the other hand, we can pick $A\cap \{2n/5+1,\ldots,n\}$ in at most $2^{3n/5}$ ways. We thus have  
\begin{equation}\label{eq:Janson-or-monochromatic-I}
\card{\cI_1(S)} \le e^{-k/96\el}2^{n/5}\cdot 2^{3n/5}=e^{-k/96\el}2^{4n/5}.
\end{equation}

We proceed to bound $\card{\cI_2(S)}$. For each subset $T \subseteq V(G)$ so that $T$ contains at least $k/8$ edges of $G$, we define $\cI_2(S,T)$ to be the collection of sets $A\in \cI_2(S)$ with $A\cap V(G)=T$. We see immediately that $\cI_2(S)=\bigcup_{T } \cI_2(S,T)$, and so the task is now to estimate $\card{\cI_2(S,T)}$. Observe that a set $A\in \cI_2(S,T)$ is uniquely determined by the intersection of $A$ and $\{2n/5+1,\ldots,n\}$. For this reason we fix $S$ and $T$, and bound the number of ways to choose $A\cap\{2n/5+1,\ldots,n\}$. Since $G$ has maximum degree at most $2\el$, we may select $k/16\el$ disjoint edges in $T=A\cap V(G)$, say $\{x_i,x_i+s_i\}$ with $1\le i \le k/16\el$. Let $B=\{x_i+s_i:1\le i \le k/16\el\}$. Then, $s_i \in A_1\cap \left[\left(\tfrac15-\tfrac{1}{100}\right)n,\tfrac15 n\right]$ by Observation (i), and so $x_i\in A_1\cap \left[\tfrac15 n,\left(\tfrac15+\tfrac{1}{100}\right)n\right]$ due to property $(\al)$ and the fact that $x_i+s_i \le 2n/5$. Hence $B\subseteq 2A_1 \cap \left[\left(\tfrac25-\tfrac{1}{100}\right)n,\tfrac25 n \right]$.
Since $A_1$ is sum-free and $B\subseteq A$, this forces  
\begin{equation}\label{eq:Janson-or-monochromatic-II}
B \subseteq A_2 \cap \left[\left(\tfrac25-\tfrac{1}{100}\right)n,\tfrac25 n\right].
\end{equation}
We thus have
\begin{equation}\label{eq:Janson-or-monochromatic-III}
2B \subseteq \left[\left(\tfrac45-\tfrac{1}{50}\right)n,\tfrac45 n\right], \ \text{and} \ \card{2B} \ge 2\card{B}-1 \ge k/16\el.
\end{equation}
This suggests us splitting $\cI_2(S,T)=\cI'_2(S,T)\dcup \cI^{''}_2(S,T)$, where $\cI'_2(S,T)$ contains every set $A\in \cI_2(S,T)$ with $\card{A\cap 2B} \le \card{2B}/4$.

We consider a set $A \in \cI'_2(S,T)$. Since $\card{A\cap 2B} \le \card{2B}/4$ by the definition of $\cI'_2(S,T)$, we may pick $A\cap 2B$ from the family of all subsets of $2B$ in at most $2^{H(1/4)\card{2B}}$ ways. Thus, noting that $2B\subseteq \{2n/5+1,\ldots,n\}$ by \eqref{eq:Janson-or-monochromatic-III} and that $A$ is determined by $A\cap \{2n/5+1,\ldots,n\}$, we have
\begin{equation}\label{eq:Janson-or-monochromatic-IV}
\card{\cI_2'(S,T)} \le 2^{H(1/4)\card{2B}}\cdot 2^{3n/5-\card{2B}} \overset{\eqref{eq:Janson-or-monochromatic-III}}{\le} 2^{3n/5-k/90\el}.
\end{equation}

Suppose now that $A\in \cI^{''}_2(S,T)$. Evidently there are at most $2^{2n/5}$ ways to choose $A\cap \{2n/5+1,\ldots,4n/5\}$. We shall fix this set and bound the number of possibilities for $A\cap \{4n/5+1,\ldots,n\}$. As $2B \subseteq \{2n/5+1,\ldots,4n/5\}$ by \eqref{eq:Janson-or-monochromatic-III}, $S+(A\cap 2B)$ is already determined. Moreover, it follows from property \eqref{eq:Janson-or-monochromatic-II} and Observation (ii) that $A\cap \{4n/5+1,\ldots,n\}$ and $S+(A\cap 2B)$ are two disjoint subsets of $\{4n/5+1,\ldots,n\}$. Hence there are at most $2^{n/5-\card{S+(A\cap 2B)}} \le 2^{n/5-\card{A\cap 2B}}$ possible outcomes for the set $A\cap \{4n/5+1,\ldots,n\}$. Therefore, we get the estimate
\begin{equation}\label{eq:Janson-or-monochromatic-V}
\card{\cI_2^{''}(S,T)} \le 2^{2n/5}\cdot 2^{n/5-\card{A\cap 2B}} \le 2^{3n/5-\card{2B}/4} \overset{\eqref{eq:Janson-or-monochromatic-III}}{\le} 2^{3n/5-k/64\el},
\end{equation}
in which the second inequality follows from the definition of $\cI_2^{''}(S,T)$.

Combining inequalities \eqref{eq:Janson-or-monochromatic-IV} and \eqref{eq:Janson-or-monochromatic-V} gives
\begin{equation}\label{eq:Janson-or-monochromatic-VI}
\card{\cI_2(S)}=\sum_{T \subseteq \{n/5+1,\ldots 2n/5\}}\left(\card{\cI_2'(S,T)}+\card{\cI_2^{''}(S,T)}\right) \le 2^{1-k/90\el}2^{4n/5}.
\end{equation}

Finally there are at most $\left(\frac{e^2k}{\el^2}\right)^{\el}$ choices for $S \in\cS(k,\el)$ by Lemma \ref{lem:restricted-partitions}, and hence, using \eqref{eq:Janson-or-monochromatic-I} and \eqref{eq:Janson-or-monochromatic-VI}, we can bound the number of sets $A$ from above by
\begin{align*}
\sum_{k \ge \el^2/\de^2}\sum_{S\in \cS(k,\el)}\left(e^{-k/96\el}+2^{1-k/90\el} \right)2^{4n/5} &\le \sum_{k\ge \el^2/\de^2}4\cdot\left(\frac{e^2k}{\el^2}\right)^{\el}e^{-k/130\el}2^{4n/5}\\
& \le 1040 \el \left(\frac{e^2}{\de^2}\right)^{\el}e^{-\el/130\de^2}2^{4n/5} \le e^{-\el}2^{4n/5},
\end{align*}
where the second inequality holds since $g(x)=x^{a}e^{-bx}$ is  decreasing on $[a/b,\infty)$ and $g(x+1/b)<g(x)/2$ for $x \ge 4a/b$. (Note that we have $\el^2/\de^2 \ge 4 \el \cdot 130\el$ since $\de>0$ is sufficiently small.)
\end{proof}

\begin{claim}\label{claim:counting-reduction-III}
For a given fixed $\ell \in \bN$, there are at most $e^{-\el}2^{4n/5}$ sets $A\in \SF_2(n)$ of type $(b)$ that satisfy $(\al)$ and with 
\begin{itemize}
\item[$(\be_1)$] $S(A)\in \cS(k,\el)$ for some $k\le \el^2/\de^2$;
\item[$(\be_2)$] $\card{S(A)+S(A)} \ge \card{S(A)}/\de^2$.
\end{itemize}
\end{claim}
\begin{proof}
The proof is similar in spirit to that of Claim \ref{claim:counting-reduction-II}. Fixing an integer $k$ with $k\le \el^2/\de^2$ and a set $S\in \cS(k,\el)$ with $\card{2S} \ge \el/\de^2$, we denote by $\cI(S)$ the collection of all sets $A \in \SF_2(n)$ of type $(b)$ that satisfy $(\al)$ and with $S(A)=S$. Further partition $\cI(S)=\cI_1(S)\dcup\cI_2(S)$, where $\cI_1(S)$ consists of all sets $A\in \cI(S)$ with $\card{A\cap 2S} \le \card{2S}/4$. 

We first count $\cI_1(S)$. Notice that $2S\subseteq \{n/5+1,\ldots,2n/5\}$ due to Observation (i), and $\card{A\cap 2S}\le \card{2S}/4$ by the definition of $\cI_1(S)$. From this we deduce that there are no more than $2^{n/5-\card{2S}}2^{H(1/4)\card{2S}}$ choices for $A\cap \{n/5+1,\ldots,2n/5\}$. Since we can take $A\cap \{2n/5+1,\ldots,n\}$ in at most $2^{3n/5}$ possible ways, it follows that
\begin{equation}\label{eq:reduction-III-1}
\card{\cI_1(S)} \le 2^{n/5-\card{2S}}2^{H(1/4)\card{2S}}\cdot 2^{3n/5} \le 2^{4n/5-\el/6\de^2}
\end{equation}
for $\card{2S} \ge \el/\de^2$.

We next deal with $\cI_2(S)$. For each subset $T\subseteq 2S$ with $\card{T} \ge \card{2S}/4$, we define $\cI_2(S,T)$ to be the collection of sets $A\in \cI_2(S)$ with $A\cap 2S=T$. We shall fix such a set $T$ and further partition $\cI_2(S,T)=\cI'_2(S,T)\dcup \cI^{''}_2(S,T)$, in which $\cI'_2(S,T)$ consists of sets $A \in \cI_2(S,T)$ with $\card{A\cap 2T} \le \card{2T}/4$.
Note that 
\begin{equation}\label{eq:reduction-III-2}
\card{2T} \ge \card{T} \ge \card{2S}/4 \ge \el/4\de^2.
\end{equation} 

Suppose first that $A \in\cI'_2(S,T)$. Then $\card{A\cap 2T}\le \card{2T}/4$ by the definition of $\cI'_2(S,T)$, and so we can choose $A\cap 2T$ in at most $2^{H(1/4)\card{2T}}$ ways. Moreover, by Observation (i) we have  
$2S \subseteq \{n/5+1,\ldots,2n/5\}$ and $2T \subseteq 4S \subseteq \{2n/5+1,\ldots,4n/5\}$. So there are at most $2^{4n/5-\card{2S}-\card{2T}}$ possibilities for $A\setminus \left(S\cup 2S\cup 2T\right)$. (Recall that $S=A\cap [n/5]$.) We therefore obtain
\begin{equation}\label{eq:reduction-III-3}
\card{\cI'_2(S,T)} \le 2^{H(1/4)\card{2T}}\cdot 2^{4n/5-\card{2S}-\card{2T}}\overset{\eqref{eq:reduction-III-2}}{\le} 2^{4n/5-\card{2S}-\el/22\de^2}.
\end{equation}

Suppose now that $A\in \cI^{''}_2(S,T)$. Since $2S \subseteq \{n/5+1,\ldots,2n/5\}$ by Observation (i), and the sets $S=A\cap [n/5]$ and $T=A\cap 2S$ have been chosen, we see that $A$ is uniquely determined by $A\cap \left(\{n/5+1,\ldots,4n/5\}\setminus 2S\right)$ and $A\cap \{4n/5+1,\ldots,n\}$. We can trivially bound the number of choices for $A\cap \left(\{n/5+1,\ldots,4n/5\}\setminus 2S\right)$ by $2^{3n/5-\card{2S}}$. We shall fix this set and bound the number of ways to choose $A\cap \{4n/5+1,\ldots,n\}$. Note that fixing $A\cap [4n/5]$ determines $S+(A\cap 2T)$. Furthermore, we know from Observation (i) that $T=A\cap 2S$ is contained in $A_2\cap \left[\left(\tfrac25-\tfrac{1}{50}\right)n,\tfrac25 n\right]$, and consequently $A\cap \{4n/5+1,\ldots,n\}$ and $S+(A\cap 2T)$ are disjoint subsets of $\{4n/5+1,\ldots,n\}$ due to Observation (ii). 
Hence we can assign $A\cap \{4n/5+1,\ldots,n\}$ in at most $2^{n/5-\card{S+(A\cap 2T)}} \le 2^{n/5-\card{A\cap 2T}}$ possible ways, as $S\ne \emptyset$. Putting everything together we get
\begin{equation}\label{eq:reduction-III-4}
\card{\cI^{''}_2(S,T)} \le 2^{3n/5-\card{2S}}\cdot 2^{n/5-\card{A\cap 2T}} \le 2^{4n/5-\card{2S}-\card{2T}/4} \overset{\eqref{eq:reduction-III-2}}{\le} 2^{4n/5-\card{2S}-\el/16\de^2},
\end{equation}
where the second inequality holds since $\card{A\cap 2T} \ge \card{2T}/4$ by the definition of $\cI^{''}_2(S,T)$. 

Using inequalities \eqref{eq:reduction-III-3} and \eqref{eq:reduction-III-4} yields
\begin{equation}\label{eq:reduction-III-5}
\card{\cI_2(S)}=\sum_{T\subseteq 2S}\left(\card{\cI'_2(S,T)}+\card{\cI^{''}_2(S,T)}\right) \le 2^{1-\el/22\de^2}2^{4n/5}.
\end{equation}

Finally adding inequalities \eqref{eq:reduction-III-1} and \eqref{eq:reduction-III-5}, and summing over all $S$, we get the following bound on the number of sets $A$:
\begin{align*}
\sum_{k \le \el^2/\de^2}\sum_{S \in \cS(k,\el)}\left(2^{-\el/6\de^2}+2^{1-\el/22\de^2}\right)2^{4n/5}&\le \sum_{k\le \el^2/\de^2}\left(\frac{e^2k}{\el^2}\right)^{\el}2^{2-\el/22\de^2}2^{4n/5}\\
& \le \frac{\el^2}{\de^2} \cdot \left(\frac{e^2}{\de^2}\right)^{\el}2^{2-\el/22\de^2}2^{4n/5} \le e^{-\el}2^{4n/5},
\end{align*}
where the first inequality holds since $\card{\cS(k,\el)} \le \left(\frac{e^2k}{\el^2}\right)^{\el}$ due to Lemma \ref{lem:restricted-partitions}.
\end{proof}

The following claim now completes the proof of Lemma \ref{lem:continuous-containers}.
\begin{claim}\label{claim:counting-reduction-IV}
There exists an absolute constant $\el_0$ so that for every integer $\el \ge \el_0$ there are at most $e^{-\el/5}2^{4n/5}$ sets $A\in \SF_2(n)$ of type $(b)$ which satisfy $(\al)$ and with 
\begin{itemize}
	\item[$(\ga_1)$] $S(A)\in \cS(k,\el)$ for some $k\le \el^2/\de^2$;
	\item[$(\ga_2)$] $\card{S(A)+S(A)} \le \card{S(A)}/\de^2$.
\end{itemize}
\end{claim}
\begin{proof}
This is the most difficult case, and we shall have to count more carefully, using Lemma \ref{lem:Green-Morris}.
For each $k \in \bN$ and $\la>0$, let $\cS_{(\la)}(k,\el)$ denote the collection of sets $S\in \cS(k,\el)$ such that
\[
\la \card{S} \le \card{2S} \le (1+\de)\la\card{S}.
\]
Given $S\in \cS_{(\la)}(k,\el)$, we denote by $\cI(S)$ the collection of all sets $A\in \SF_2(n)$ of type $(b)$ that satisfy $(\al)$ and with $S(A)=S$. It is not hard to see that the number of sets $A$ that satisfies the hypothesis of Claim \ref{claim:counting-reduction-IV} is bounded from above by $\sum\card{\cI(S)}$, where the sum is taken over all triples $(k,\la,S)$ with $k \le \el^2/\de^2$, $\la=(2-\de)(1+\de)^{i}$ for some integer $i$ with $0\le i\le \frac{3}{\de} \ln \frac{1}{\de}$, and $S\in \cS_{(\la)}(k,\el)$. To count $\cI(S)$, we partition $\cI(S)=\cJ(S)\dcup \cK(S)$, in which $\cJ(S)$ consists of all sets $A\in \cI(S)$ with $\card{A\cap 2S} \le \card{2S}/20$. 

We shall use Lemma \ref{lem:Green-Morris} to count the number of triples $(k,\la,S)$. As noted above, there are only $O_{\de}(\el^2)$ choices for $k$ and $\la$; this will be absorbed by the error term $2^{O(\de \el)}$. We may apply Lemma \ref{lem:Green-Morris} to $R=(1+\de)\la$, $s=\el$ and $D=k$, and conclude that there are at most $2^{O(\de \el)}\binom{(1+\de)\la \el/2}{\el}$ choices for $S \in \cS_{(\la)}(k,\el)$. (Note that $(1+\de)\la=O_{\de}(1)$, $k=O_{\de}(\el^2)$ and $\el$ is sufficiently large.)

We are now ready to estimate the sum $\sum_{(k,\la,S)}\card{\cJ(S)}$. Analysis similar to that in the proof of Claim \ref{claim:counting-reduction-III} shows
\begin{equation*}
\card{\cJ(S)}\le 2^{4n/5-\card{2S}}2^{H(1/20)\card{2S}} \le 2^{4n/5-2\la \el/3}
\end{equation*}
since $\card{2S} \ge \la \el$ for all $S\in \cS_{(\la)}(k,\el)$. Summing over all choices of $(k,\la,S)$, and recalling that $\la=(2-\de)(1+\de)^{i} \ge 2-\de$, we thus get
\begin{align}\label{reduction-III-1}\notag
\sum_{(k,\la,S)}\card{\cJ(S)} &\le \sum_{\la}2^{O(\de \el)}\binom{(1+\de)\la \el/2}{\el}2^{4n/5-2\la \el/3}\\
& \le \sum_{\la}2^{O(\de \el)}2^{(1+\de)\la \el/2}2^{4n/5-2\la \el/3} \le 2^{4n/5-0.33\el}.
\end{align}

We proceed to bound the sum $\sum_{(k,\la,S)}\card{\cK(S)}$.
For each $p\in \bN$ and $\mu>0$, let $\cT^{(\mu)}(S,p)$ be the collection of sets $T \subseteq 2S$ with
\[
\card{T}=p, \ \text{and} \ \mu \card{T} \le \card{2T} \le (1+\de)\mu\card{T}.
\]
For any set $T\in \cT^{(\mu)}(S,p)$ and any integer $q$ with $0 \le q \le \card{2T}$, let $\cK(T,q)$ stand for the collection of  those sets $A\in \cK(S)$ with $A\cap 2S=T$ and $\card{A\cap 2T}=q$. From the definition of $\cK(S)$, we know that $\card{2S}/20 \le \card{T} \le \card{2S}$ (otherwise $\cK(T,q)=\emptyset$), and so $\la \el/20 \le p \le (1+\de)\la \el$. Moreover, as $\card{2S} \le \card{S}/\de^4$ by our choice of $\la$, Lemma \ref{lem:Plunnecke} implies $\card{4S} \le \card{S}/\de^{16}$, giving $\card{2T} \le \card{4S} \le \card{2S}/\de^{16} \le 20\card{T}/\de^{16}$. Accordingly we only need to care about those $\mu$ so that $\mu=(2-\de)(1+\de)^j$ for some integer $j$ with $0 \le j \le \frac{17}{\de} \ln \frac{1}{\de}$.
Summarizing, we have 
\[
\card{\cK(S)} \le \sum_{(p,\mu,T,q)}\card{\cK(T,q)},
\] 
where the sum is over all quadruples $(p,\mu,T,q)$ such that $\la \el/20 \le p \le (1+\de)\la \el$, $\mu=(2-\de)(1+\de)^j$ for some integer $j$ with $0 \le j  \le \frac{17}{\de} \ln \frac{1}{\de}$, $T\in \cT^{(\mu)}(S,p)$, and $q \le \card{2T}$.

From the previous discussion, we deduce that there are only $O_{\de}(p)$ choices for $p$ and $\mu$; this will be absorbed by the error term $2^{O(\de p)}$. Using Lemma \ref{lem:Green-Morris} with $R=(1+\de)\mu$, $s=p$ and $D=2k$, we find that there are at most $2^{O(\de p)}\binom{(1+\de)\mu p/2}{p}$ choices for $T\in \cT^{(\mu)}(S,p)$. 
Since $q \le \card{2T} \le (1+\de)\mu p$, we have only $O_{\de}(p)$ possibilities for $q$, and this will also be absorbed by the error term $2^{O(\de p)}$.  

We are reduced to enumerating $\cK(T,q)$ for fixed $T\in \cT^{(\mu)}(S,p)$ and $q \le (1+\de)\mu p$. Since $\card{A\cap 2T}=q$, there are at most $\binom{\card{2T}}{q}$ choices for $A\cap 2T$. In addition, since $2S$ and $2T$ are disjoint subsets of $\{2n/5+1,\ldots,4n/5\}$ due to Observation (i), we can allocate $A\cap \left(\{2n/5+1,\ldots,n\}\setminus (2S\cup 2T)\right)$ in at most $2^{3n/5-\card{2S}-\card{2T}}$ possible ways. Furthermore, specifying $A\cap [4n/5]$ determines $S+(A\cap 2T)$. As $A\cap \{4n/5+1,\ldots, n\}$ and $S+(A\cap 2T)$ are disjoint subsets of $\{4n/5+1,\ldots, n\}$ by Observation (ii), this implies that there are at most $2^{n/5-\card{S+(A\cap 2T)}} \le 2^{n/5-q}$ possibilities for $A\cap \{4n/5+1,\ldots,n\}$. 
Here we evaluate $\card{S+(A\cap 2T)} \ge \card{A\cap 2T}=q$ for $S\ne \emptyset$. Therefore, recalling that $\card{2S} \ge \la \el$ and $\mu p \le \card{2T} \le (1+\de)\mu p$, we get
\begin{equation*}
\card{\cK(T,q)} \le\binom{\card{2T}}{q}\cdot 2^{3n/5-\card{2S}-\card{2T}} \cdot 2^{n/5-q} \le \binom{(1+\de)\mu p}{q} 2^{4n/5-\la\el-\mu p-q}.
\end{equation*}

From what has already been proved we may bound $\sum_{(k,\la,S)} \card{\cK(S)}$ from above by
\begin{align*}
&\sum_{(k,\la,S)}\sum_{(p,\mu,T,q)}\card{\cK(T,q)} \\
&\le \sum_{(\la,p,\mu,q)}2^{O(\de \el)}\binom{(1+\de)\la \el/2}{\el}\cdot 2^{O(\de p)}\binom{(1+\de)\mu p/2}{p}\cdot\binom{(1+\de)\mu p}{q} 2^{4n/5-\la\el-\mu p-q}\\
& \le 2^{4n/5} \cdot \max_{(\la,p,\mu,q)}\left \{\binom{(1+\de)\la \el/2}{\el}2^{-\la \el} \cdot 2^{O(\de p)}\binom{(1+\de)\mu p/2}{p}2^{-\mu p}\cdot \binom{(1+\de)\mu p}{q}2^{-q} \right\},
\end{align*}
where in the last inequality we used the fact that the term $2^{O(\de \el)}$ is absorbed by the error term $2^{O(\de p)}$. We shall deploy the entropy estimate \eqref{entropy} to control the last expression. For abbreviation, set $x=(1+\de)\la/2$, $y= (1+\de)\mu/2$, $z=q/(1+\de)\mu p$, and $\mathbb{D}=\{(x,p,y,z)\in \bR^4: x \ge 1, 0 \le p \le 2x\el, y \ge 1, 0 \le z \le 1\}$.
Recalling that $\la, \mu \ge 2-\de$, $0\le p \le (1+\de)\la \el$ and $0 \le q \le (1+\de)\mu p$, we then have $(x,p,y,z) \in \mathbb{D}$. Now using inequality \eqref{entropy} and simplifying yields  
\begin{equation*}
\sum_{(k,\la,S)}\card{\cK(S)} \le  2^{4n/5} \cdot \max_{(x,p,y,z)\in \mathbb{D}}2^{h(x,p,y,z)},
\end{equation*}
where $h(x,p,y,z):=\left(xH(\tfrac1x)-\frac{2x}{1+\de}\right)\el+\left(yH(\tfrac1y)-\frac{2y}{1+\de}+O(\de)\right)p+ 2yp\cdot (H(z)-z)$.
A straightforward but slightly tedious calculation shows that the maximum value of $h(x,p,y,z)$ on $\mathbb{D}$ is $\left(\log_2\left(\frac{81}{115}\right)+O(\de)\right)\el \approx -0.505 \el$, attained at $z=\frac13$, $y=\frac{16}{7}+O(\de)$, $p=2x\el$ and $x=\frac{196}{115}+O(\de)$.\footnote{We can solve this optimisation problem backwardly using the following simple facts. Firstly, the function $f(z)=H(z)-z$ achieves its maximum at $z=1/3$.
Secondly, given $\rho>0$, the function $g(t)=tH(\tfrac1t)-\rho t$ is maximised at $t=2^{\rho}/(2^{\rho}-1)$.} Hence
\begin{equation}\label{reduction-III-2}
\sum_{(k,\la,S)}\card{\cK(S)} \le  2^{4n/5}\cdot 2^{-0.5\el}=2^{4n/5-0.5\el}.
\end{equation}

Finally adding inequalities \eqref{reduction-III-1} and \eqref{reduction-III-2}, and summing over all triples $(k,\la,S)$, we conclude that the number of sets $A$ is at most
\[
\sum_{(k,\la,S)} \card{\cI(S)} \le \sum_{(k,\la,S)}\left(\card{\cJ(S)}+\card{\cK(S)}\right) \le \left(2^{-0.33\el}+2^{-0.5\el}\right)2^{4n/5} \le e^{-\el/5}2^{4n/5}.\qedhere
\]
\end{proof}
The proof of Lemma \ref{lem:continuous-containers} is at long last complete.
\end{proof}

\section{Concluding remarks}\label{sec:remarks}
In this paper we studied the general structure of large sum-free sets of integers. From this we obtained a good bound on the total number of $2$-wise sum-free subsets of $[n]$. It is likely that our methods extend to give an asymptotic formula for this number, but we do not pursue this here. We close with some remarks
and possible directions for further research.

\subsection*{Sets with small difference constant}
The main open problem is to determine the critical density
threshold at which Theorem \ref{thm:structure} ceases to hold. Note that in the theorem, the value for $c$ given by our argument is something like $10^{-6} c_{\ref{thm:Jin}}^2$, where $c_{\ref{thm:Jin}}$ is the absolute positive constant from Jin's inverse theorem (Lemma \ref{thm:Jin}). Note that Jin obtained his result via non-standard analysis, and thus no explicit value of $c_{\ref{thm:Jin}}$ can be extracted from his proof.
Using the following conjecture instead of Lemma \ref{thm:Jin}, we would certainly get a reasonable value for $c$.

\begin{conj}\label{conj:small-difference}
There exists a natural number $K$ such that for any finite set of integers $A$ so that $\card{A} \ge K$ and $\card{A-A}=3\card{A}-3+r$ for some integer $r$ with $0<r <\tfrac13\card{A}-2$, one of the following properties holds:
\begin{itemize}
	\item[(i)] $A$ is a subset of an arithmetic progression of length $2\card{A}-1+2r$;
	\item[(ii)] $A \subseteq P_1\cup P_2$ for some arithmetic progressions $P_1,P_2$ with common step and $\card{P_1}+\card{P_2} \le \card{A}+r$.   
\end{itemize}
\end{conj}

We remark that the sumset version of Conjecture \ref{conj:small-difference}  was proposed by Freiman \cite{Freiman60}. The following example shows that the condition $r<\tfrac13 \card{A}-2$ is necessary. 

\begin{example}
Let $y\ge 4x$, and consider the set $A=\{0,y,2y\}+[0,x-1]$. We have $A-A=\{0,\pm y,\pm 2y\}+[-x+1,x-1]$, and so $\card{A-A}=10x-5=(3\card{A}-3)+(\tfrac13 \card{A}-2)$. 
But $A$ is neither a subset of an arithmetic progression of length $(2\card{A}-1)+2\cdot (\tfrac13\card{A}-2)$ nor a subset of an union of two arithmetic progressions of total length $\card{A}+(\tfrac13\card{A}-2)$. 
\end{example}

It is worth mentioning that Eberhard, Green and Manners \cite{EGM14} provided a rough structure theorem for sets of integers of difference constant less than $4$. Specifically, they proved that if $A$ is a subset of $\bZ$ with $\card{A-A}\le (4-\eps)\card{A}$ then $A$ has density at least $\tfrac12 +2^{-1000}\eps$ on some arithmetic progression of length $\gg_{\eps} \card{A}$. They then used this result to show the existence of a set of $n$ positive integers with no sum-free subset of size greater than $\tfrac13 n+o(n)$, answering a famous question of Erd\H{o}s \cite{Erdos65} from 1965.

\subsection*{Union of intersecting families}
One can pursue the following general questions for any {\bf monotone property} $\cP$:
\begin{itemize}
	\item[(i)] What is the maximum size of a union of $r$ objects with property $\cP$?
	\item[(ii)] How many objects which can be partitioned into $r$ subobjects having property $\cP$ are there?
\end{itemize}
In this paper, we addressed the second question for the sum-free property. In what follows, we shall single out another monotone property for further research. 

A family of sets is called {\bf intersecting} if it does not contain two disjoint sets. Given a positive integer $r$, a family $\cF$ is said to be {\bf $r$-wise intersecting} if there exists a partition of $\cF$ into $r$ intersecting families.
Let $\cI_r(n,k)$ denote the collection of all $r$-wise intersecting families $\cF\subseteq \binom{[n]}{k}$.
The celebrated Erd\H{o}s-Ko-Rado theorem from 1961 states that for $n\ge 2k$ the largest member of $\cI_1(n,k)$ has size $\binom{n-1}{k-1}$. Recently Ellis and Lifshitz \cite{EL17} considered the problem, first raised by Erd\H{o}s \cite{Erdos73}, of determining the maximum possible size of a family in $\cI_r(n,k)$ when $r\ge 2$. Specifically, they showed $\card{\cF} \le \binom{n}{k}-\binom{n-r}{k}$ for any $\cF \in \cI_r(n,k)$ provided that $r\ge 2$ and $n\ge 2k+C(r)k^{2/3}$, with equality holds if and only if $\cF=\left\{F\in \binom{[n]}{k}:F\cap R\ne \emptyset\right\}$ for some $R\in \binom{[n]}{r}$. In the case $r=2$, this significantly improves a previous result due to Frankl and F{\"u}redi \cite{FF86}. It would be interesting to determine whether  $C(r)k^{2/3}$ is the best possible error term. Note that an example given by Frankl and F{\"u}redi \cite{FF86} shows that this term cannot be reduced to below $\sqrt{k}$.

The problem of enumerating $\cI_1(n,k)$ was first investigated by Balogh, Das, Delcourt, Liu and Sharifzadeh \cite{BDDLS15}. Building on the  of Balogh et al., Frankl and Kupavskii
\cite{FK17} and, independently, Balogh, Das, Liu, Sharifzadeh and Tran \cite{BDLST17} established the asymptotic formula $\card{\cI_1(n,k)}=(n+o(1))2^{\binom{n-1}{k-1}}$ for $n \ge 2k+3\sqrt{k \ln k}$. Motivated by this result and the theorem of Ellis and Lifshitz, we make the following conjecture.

\begin{conj}
$\card{\cI_r(n,k)}=\left(\binom{n}{r}+o(1)\right)2^{\binom{n}{k}-\binom{n-r}{k}}$ for $r\ge 2$ and $n\ge 2k+C(r)k^{0.9}$, where the term $o(1)$ tends to $0$ as $n\rightarrow \infty$.
\end{conj}

\section*{Acknowledgement} 
The author was supported by the  Alexander Humboldt Foundation, and by the GACR grant GJ16-07822Y, with institutional support RVO:67985807. He would like to thank Jan Hladky and Phuong Dao for fruitful discussions. He is also grateful to an anonymous referee whose suggestions helped improve and clarify the manuscript.

\appendix
\section{Missing proofs from Section \ref{sec:structure}}
In this appendix, we give the proofs of Lemmas \ref{lem:long-interval}, \ref{lem:summation} and \ref{lem:bootstrap}.

\begin{proof}[Proof of Lemma \ref{lem:long-interval}]
	(i) As $m \in A$, we must have $\card{A\cap \{i,m+i\}} \le 1$ for all $i\in [u,v]$. Hence, by the union bound, we obtain
	\[
	\card{A\cap \left([u,v]\cup [u+m,v+m]\right)} \le \sum_{u \le i \le v}\card{A\cap \{i,m+i\}} \le v-u+1. 
	\]
	
	(ii) Using part (i) with $v=u+m-1$, we find 
	\[
	\card{A\cap [u,u+2m-1]} \le (u+m-1)-u+1=m.
	\]
	
	(iii) Write $v-u=2km+r$, where $k,r \in \bN$ and $0\le r<2m$. It follows easily from part (ii) that $\card{A\cap [u,u+2km-1]}\le km$.
	If $r \le m-1$, then we can trivially evaluate 
	\[
	\card{A\cap [u+2km,u+2km+r]} \le r+1 \le \tfrac12(r+m+1).
	\] 
	If $m\le r<2m$, part (ii) gives $\card{A\cap [u+2km,u+2km+r]} \le m <(r+m+1)/2$. Hence in either case, we always have $\card{A \cap [u,v]} \le km+(r+m+1)/2=(v-u+m+1)/2$.  
\end{proof}

\begin{proof}[Proof of Lemma \ref{lem:summation}]
	We begin by showing that the set $[x,x+k-1]\setminus (A+B)$ has at most $2\eps k$ elements for each integer $x \in [b_1,b_{\el}+1]$. Indeed let $i\in [\el]$ be the largest integer such that $b_i \le x$. For convenience, set $b_{\el+1}=b_{\el}+1$. From the definition of $i$ and the fact that $b_{i+1}-b_i \le k$, we find $[x,x+k-1] \subseteq [b_i,b_{i+1}+k-1]=\{b_i,b_{i+1}\}+[0,k-1]$. Moreover, since $\card{A}\ge (1-\eps)k$, there are at most $2k\eps$ elements in $\{b_i,b_{i+1}\}+[0,k-1]$ which do not belong to $\{b_i,b_{i+1}\}+A$. Hence $[x,x+k-1]$ contains only elements of $A+B$, with at most $2k\eps$ exceptions, as claimed.
	
	Finally, because $[b_1,b_{\el}+k]$ can be covered by at most $(k+b_{\el}-b_1+1)/k+1$ intervals of the form $[x,x+k-1]$ with $x\in [b_1,b_{\el}+1]$, we find
	\[
	\card{A+B} \ge (k+b_{\el}-b_1+1)-2k\eps\cdot \left(\frac{k+b_{\el}-b_1+1}{k}+1\right) \ge (1-4\eps)(k+b_{\el}-b_1+1). \qedhere
	\]
\end{proof}

\begin{proof}[Proof of Lemma \ref{lem:bootstrap}] 
	(i) A proof of this result can be found in \cite[Lemma 2.2]{DFST99}. 
	
	(ii) Denote $s=\card{A}$. We wish to show that $x\in 2A$ for each $x \in [2k-2s+2,2s-2]$. Since $[2k-2s+2,2s-2]=[k-s+1,s-1]+[k-s+1,s-1]$, one has $x=y+z$ for some integers $y,z \in [k-s+1,s-1]$. Note that $y+i, z-i \in [0,k]$ for every integer $i \in [-k+s-1,k-s+1]$, and $\card{[0,k]\setminus A}\le k-s+1$. Thus by the pigeonhole principle, there exists $j \in [-k+s-1,k-s+1]$ so that $y+j, z-j\in A$. We then have $x=(y+j)+(z-j)\in 2A$.
\end{proof}

\end{document}